\documentclass{article}

\usepackage[utf8]{inputenc}
\usepackage[a4paper,left=2.5cm,right=2.5cm,top=2cm,bottom=2cm]{geometry}
\usepackage{amsmath}
\usepackage{stmaryrd}
\usepackage{amssymb}
\usepackage{thmbox} 
\usepackage[colorlinks, linktocpage=true]{hyperref}[breaklinks=true]
\usepackage{cleveref}
\usepackage{enumitem}
\usepackage{multicol}
\usepackage{esvect}
\usepackage[dvipsnames]{xcolor}
\usepackage{xargs}
\usepackage[colorinlistoftodos,prependcaption,textsize=tiny]{todonotes}
\usepackage{multirow}
\usepackage{bbm, dsfont}
\usepackage{shortlst}
\usepackage{diagbox}
\usepackage{makecell}
\usepackage{float}
\usepackage{rotating}
\usepackage{nicematrix}
\usepackage{tabularx}

\crefname{section}{section}{sections}
\crefname{subsection}{subsection}{subsections}
\Crefname{section}{Section}{Sections}
\Crefname{subsection}{Subsection}{Subsections}
\Crefname{figure}{Figure}{Figures}
\crefname{hyp}{hypothesis}{hypotheses}
\Crefname{hyp}{Hypothesis}{Hypotheses}
\crefname{lem}{lemma}{lemmas}
\Crefname{lem}{Lemma}{Lemmas}
\Crefname{prop}{Proposition}{Propositions}
\Crefname{thm}{Theorem}{Theorems}
\Crefname{table}{Table}{Tables}
\Crefname{defi}{Definition}{Definitions}

\newtheorem[L]{thm}{Theorem}[section]
\newtheorem[L]{defi}{Definition}[section]
\newtheorem{prop}{Proposition}[section]
\newtheorem{lem}{Lemma}[section]
\newtheorem{rem}{Remark}[section]

\makeatletter
\newcommand*{\bdiv}{%
\nonscript\mskip-\medmuskip\mkern5mu%
\mathbin{\operator@font div}\penalty900\mkern5mu%
\nonscript\mskip-\medmuskip
}
\makeatother

\renewcommand{\d}[1]{\ensuremath{\operatorname{d}\!{#1}}}

\newcommand{\seq}{\ell^1(\mathbb{N},\mathbb{R}_+)}

\newcommand{\seqseq}{\ell^1(\mathbb{N}^2,\mathbb{R}_+)}

\newcommand{\cl}[1]{\mathrm{cl}(#1)}
\newcommand{\trmat}[1]{\boldsymbol{Q}(#1)}

\title{A partition method for bounding continuous-time Markov chain models of general reaction network}
\author{Guillaume Ballif\footnote{Inria Saclay, 91120 Palaiseau, France. Mail: \href{mailto:guillaume.ballif@inria.fr}{guillaume.ballif@inria.fr}}, Laurent Pfeiffer\footnote{Université Paris-Saclay, CNRS, CentraleSupélec, Inria, Laboratoire des signaux et systèmes, and Fédération de Mathématiques de CentraleSupélec, 91190, Gif-sur-Yvette, France. Mail: \href{mailto:laurent.pfeiffer@inria.fr}{laurent.pfeiffer@inria.fr}}, and Jakob Ruess\footnote{Inria Saclay, 91120 Palaiseau, France. Mail: \href{mailto:jakob.ruess@inria.fr}{jakob.ruess@inria.fr}}}
\date{\today}

\begin{document}

\maketitle

\textbf{Abstract:} In this work, we present a general method to establish properties of multi-dimensional continuous-time Markov chains representing stochastic reaction networks. This method consists of grouping states together (via a partition of the state space), then constructing two one-dimensional birth and death processes that lower and upper bound the initial process under simple assumptions on the infinitesimal generators of the processes. The construction of these bounding processes is based on coupling arguments and transport theory. The bounding processes are easy to analyse analytically and numerically and allow us to derive properties on the initial continuous-time Markov chain. We focus on two important properties: the behavior of the process at infinity through the existence of a stationary distribution and the error in truncating the state space to numerically solve the master equation describing the time evolution of the probability distribution of the process. We derive explicit formulas for constructing the optimal bounding processes for a given partition, making the method easy to use in practice. We finally discuss the importance of the choice of the partition to obtain relevant results and illustrate the method on an example chemical reaction network.\\


\section{Introduction}

Stochastic reaction networks have been widely used in many fields such as chemistry \cite{koeppl_continuous_2011}, biology \cite{elowitz_stochastic_2002,allen_introduction_2011} and ecology, medicine \cite{briggs_introduction_1998}, epidemiology \cite{anderson_infectious_1992}, economics \cite{bai_conditional_2011} to model complex and often multi-dimensional random phenomena. Mathematically, it is straightforward to represent such models as continuous-time Markov chains (CTMCs) but the multi-dimensional nature of the real process greatly complicates the study of properties of the corresponding CTMC (such as explosion in finite time and the behaviour when time tends to infinity) and the numerical simulation of trajectories. For example, the existence of a stationary distribution is a crucial property in the study of the long-term behaviour but whether or not a given CTMC admits a stationary distribution is often hard to say \textit{a priori}. This question has been studied extensively and several methods have been developed to address it: coupling techniques \cite{lindvall_lectures_2002}, conditions on the semi-group or the study of the generator spectrum \cite{anderson_continuous-time_1991,davies_linear_2007}, extension of the Perron-Frobenius theorem \cite{seneta_non-negative_1981} or existence of a Lyapunov function \cite{meyn_stability_1993,champagnat_general_2023}.  However, the size of the state space and the assumptions to be verified make these methods difficult to apply systematically to multidimensional CTMCs.

From a numerical point of view, the time evolution of the probability distribution of a CTMC is described by a master equation (also called Kolmogorov equation) composed of an infinite system of ordinary differential equations. The size of the state space often makes solving the master equation numerically impossible due to the exponential increase in computation time with the state space dimension. Two classes of methods have been developed in the literature to circumvent this problem: Monte Carlo methods and approximation methods.

Monte Carlo methods \cite{gillespie_exact_1977} are based on the simulation of a large number of trajectories of the process with the stochastic simulation algorithm (SSA) \cite{gillespie_exact_1977} also known as the Gillespie algorithm in order to reconstruct the probability distribution of the CTMC. Various improvements have been developed over the past few years to improve the precision and computation speed of the method \cite{gillespie_approximate_2001,albi_efficient_2022}. However, a large number of trajectories of the process is required to obtain good approximations of the probability distribution, which is very costly in terms of computing time.

The second type of methods aims to solve an approximation of the master equation \cite{andreychenko_distribution_2017,kuntz_stationary_2019,einkemmer_hierarchical_2024}. One of the most commonly used methods is the truncation of the state space to replace numerical solving of the master equation by that of a linear system of ordinary differential equations \cite{munsky_finite_2006,wolf_solving_2010} but in the case of a multi-dimensional CTMC, even after truncation the state space is often still much too large for numerical solution of this linear system to be efficient or even just feasible. Moreover, most of these methods do not provide an estimate of the error caused by this approximation that can be calculated in practice.

In this work, we present a general method for constructing one-dimensional bounding processes (in a certain sense) for any CTMC under some simple assumptions on the infinitesimal generator. We can then analyse the one-dimensional bounding processes to derive properties of the multi-dimensional original CTMC that would be hard to obtain directly. Concretely, we focus on two essential aspects when studying CTMC models: the behavior of the process at infinity through the existence of a stationary distribution and the quantification of the approximation error when the infinite state space is truncated. The method for constructing the bounding processes is based on an ordered partition of the state space (that we will refer to as classes) that allows us to reduce the dimension of the CTMC by studying its behavior only on classes. Coupling arguments and transport theory are used to derive these bounding results (upper and lower bounds).

Different methods have been developed in order to reduce the state space of a CTMC to be able to derive properties on this CTMC: aggregation of states (also called lumping) for Markov chain on finite state space \cite{barreto_lumping_2011} or for some particular cases of Markov chain on an infinite state space \cite{rubino_finite_1993,ganguly_markov_2014,lu_chemical_2022}, decomposition of the reaction network into subnetworks \cite{gupta_computational_2018} or comparisons between processes through coupling \cite{fugger_majority_2024}.

To make our framework useful in practice, we derive explicit formulas for constructing the infinitesimal generator of the optimal bounding processes from assumptions on the generator of the original process. These explicit formulas are used to illustrate on an example the results that we can derive on the multi-dimensional original CTMC. The choice of optimal classes to partition the state space (i.e. leading to the best possible bounds) is also discussed by comparing the upper-bounding process and properties derived from a naive choice and those derived from a set of classes taking into account the structure of the reaction network.

The manuscript is organized as follows. \Cref{sec:upper} focuses on the construction of an upper-bounding process. The construction is done in two main steps. In Theorem \ref{thm:main_upper-bound}, we construct an upper-bounding process based on an infinitesimal generator satisfying some monotonicity properties. We provide in Proposition \ref{prop:optimal_upper} an explicit formula for such an infinitesimal generator that most closely upper-bounds the initial process.
The adaptation of these results for a lower-bounding process is presented in \Cref{sec:lower}. In \Cref{sec:appli}, the method is used to derive results on the existence of a stationary distribution and on the error made in numerically solving the master equation on truncated state spaces and all these results are illustrated on an example.

\paragraph{Notation}

We fix a countable state space $\mathcal{S}$ and a partition $(S_{\ell})_{\ell \in \mathbb{N}}$ of $\mathcal{S}$, i.e., for all $\ell \in \mathbb{N}$, $S_{\ell} \subset \mathcal{S}$ and
\begin{equation*}
\mathcal{S}= \bigcup_{\ell \in \mathbb{N}} S_{\ell},
\quad \text{and} \quad
S_{\ell} \cap S_m = \emptyset,
\quad \forall \ell \neq m.
\end{equation*}
In the following, we will assume than \textbf{for all $\ell \in \mathbb{N}$, the set $S_{\ell}$ is finite}.

\begin{defi}\label{def:fct_cl}
For any $i \in \mathcal{S}$, we will call ``class'' of $i$ the unique integer $\ell \in \mathbb{N}$ such that $i \in S_{\ell}$. We denote it by $\cl{i}$.
\end{defi}

\begin{defi}\label{def:upper_lower_processes}
Let $X= (X_t)_{t \geq 0}$ and $Y= (Y_t)_{t \geq 0}$ be two Markov chains on $\mathcal{S}$ and $\mathbb{N}$, respectively. We say that $Y$ is an upper-bounding process (of $X$) if
\begin{equation*}
Y_t \geq \cl{X_t}, \quad
\forall t \geq 0, \text{ almost surely.}
\end{equation*}
We say that $Y$ is a lower-bounding process (of $X$) if
\begin{equation*}
Y_t \leq \cl{X_t}, \quad
\forall t \geq 0, \text{ almost surely.}
\end{equation*}
\end{defi}

\vspace{0.3cm}

As we will consider Markov chains on $\mathcal{S}$, $\mathbb{N}$, and $\mathcal{S} \times \mathbb{N}$, we will make use of the following convention:
\begin{align*}
& \text{the letters $i$ and $j$ denote states in $\mathcal{S}$,} \\
& \text{the letters $k$, $\ell$, and $m$ denote states in $\mathbb{N}$.}
\end{align*}

In the article, we will abusively call \emph{matrix} any real-valued mapping with $2$ or more input variables, taking values in a countable set. For any countable set $\Sigma$, we denote as $\trmat{\Sigma}$ the set of matrices $Q= (Q_{i,j})_{i,j \in \Sigma}$ such that
\begin{itemize}
\item for all $i \in \Sigma$, for all $j \in \Sigma \backslash \{ i \}$, $Q_{i,j} \geq 0$,
\item for all $i \in \Sigma$, $\sum\limits_{j \in \Sigma \backslash \{ i \}} Q_{i,j} < \infty$,
\item for all $i \in \Sigma$, $\sum\limits_{j \in \Sigma} Q_{i,j}= 0$.
\end{itemize}
We will call such matrices \emph{transition-rate matrices}. Note that we allow the matrix $Q$ to have an infinite number of non-zero values. In the case of reaction networks, this generality allows us to consider an infinite number of reactions. Some special type of reaction networks are composed of an infinite number of reactions (such as polymerization reactions \cite{coleman_multistate_1963}). However, most reaction networks are composed of a finite number of reactions but may be composed of non-standard reactions that allow for transitions to infinitely many different states (such as the production of proteins in bursts with stochastic burst sizes) \cite{shahrezaei_analytical_2008}.

In the following, the transition-rate matrix of the Markov chain of interest will be denoted by $Q$. The aim of this article is to construct the transition-rate matrix of the upper- and lower-bounding processes that will be denoted by $\widetilde{Q}= (\widetilde{Q}_{\ell,m})_{\ell,m \in \mathbb{N}} \in \trmat{\mathbb{N}}$.
We will make use of the following notations for $Q$ and $\widetilde{Q}$:
\begin{equation} \label{eq:convention1}
\begin{array}{rll}
Q_{i,\ell} = & \! \! {\displaystyle \sum_{j \in S_\ell}} Q_{i,j}, \quad & \forall i \in \mathcal{S},\, \forall \ell \in \mathbb{N}, \\
\widetilde{Q}_{\ell,0:m}
= & \! \! {\displaystyle \sum_{k=0}^m \widetilde{Q}_{\ell,k}, } \quad & \forall \ell, m \in \mathbb{N},\\
\widetilde{Q}_{\ell,m:\infty} = & \! \! {\displaystyle \sum_{k= m}^{\infty} \widetilde{Q}_{\ell,k}, }  \quad & \forall \ell, m \in \mathbb{N}.
\end{array}
\end{equation}
These notations will also be used for general matrices (which are not necessarily transition-rate matrices).
For example, we will use the notation
\begin{equation*}
Q_{i,0:m}= \sum_{k=0}^m Q_{i,k} = \sum_{k=0}^m \sum_{j \in S_k} Q_{i,j}, \quad \forall i \in \mathcal{S},\, \forall m \in \mathbb{N}.
\end{equation*}
Note that for all $\hat{Q} \in \trmat{\Sigma}$, for all $(\ell,m) \in \Sigma^2$, $\hat{Q}_{\ell,0:m} = - \hat{Q}_{\ell,m+1:\infty}$.\\

Finally, we recall the master equation governing the time evolution of the probability distribution of a continuous-time Markov chain $X$ on $\mathcal{S}$ with transition-rate matrix $Q$:
$$ \frac{\d p}{\d t}(t,j) = \sum_{i \in \mathcal{S}} Q_{i,j}p(t,j), $$
where $p(t,i) = \mathbb{P}(X_t = i)$.
This equation can be formulated in a vector form:
$$ \frac{\d p}{\d t}(t) = p(t) Q, $$
where $p(t) = (p(t,i))_{i \in \mathcal{S}}$ is a row vector.

\section{Construction of an upper-bounding process}\label{sec:upper}

\subsection{Main result}

We state here our first main result, which gives two simple conditions on the matrix $\widetilde{Q}$ of the process $Y$ for which $Y$ is an upper-bounding process of $X$.

\begin{thm} \label{thm:main_upper-bound}
Let $Q \in \trmat{\mathcal{S}}$ and let $\widetilde{Q} \in \trmat{\mathbb{N}}$. Let $p^0$ be a distribution on $S$. Assume the two following properties:
\begin{itemize}
\item[] \begin{enumerate}[label = \textbf{\upshape (A\arabic*)}, ref = \textbf{\upshape A\arabic*}
]
\item \label{assum:second_index}
$\forall \ell$, $m \in \mathbb{N}$,~ $\forall i \in S_{\ell}$, ~$\widetilde{Q}_{\ell,0:m} \leq Q_{i,0:m}$,
\item \label{assum:first_index}
$\forall \ell$, $m \in \mathbb{N}$ such that $\ell \neq m$,~ $\widetilde{Q}_{\ell+1,0:m} \leq \widetilde{Q}_{\ell,0:m}$.
\end{enumerate}
\end{itemize}
Then there exist two continuous-time Markov chains $X=(X_t)_{t \geq 0}$ and $Y= (Y_t)_{t \geq 0}$, defined on $\mathcal{S}$ and $\mathbb{N}$ respectively, with transition-rate matrices $Q$ and $\widetilde{Q}$, such that $X_0$ is distributed according to $p^0$, and such that $Y$ is an upper-bounding process of $X$.
\end{thm}

Subsections \ref{subsec:transport} and \ref{subsec:proof} are dedicated to the proof of the theorem. The main idea of our analysis consists in constructing $(X_t,Y_t)_{t\geq 0}$ as a Markov chain on $\mathcal{S} \times \mathbb{N}$ with a transition-rate matrix $R$ satisfying certain properties (see Lemma \ref{lem:cond:R}). In particular, $R$ will be such that no transition towards a state $(i,\ell)$ such that $\cl{i} > \ell$ is possible. We will provide the reader with an explicit construction of $R$, based on a technique from transport theory (see Proposition \ref{prop:transportation}). Subsection \ref{subsec:construction_Qtilde} will be dedicated to the construction of an explicit transition-rate matrix $\widetilde{Q}$ satisfying Assumptions \ref{assum:second_index} and \ref{assum:first_index}.

We give an interpretation of the two standing assumptions in the next remarks. Note that when $m \geq \ell$, the term $\widetilde{Q}_{\ell,0:m}$ has no straightforward interpretation, as the underlying sum contains the diagonal negative term $\widetilde{Q}_{\ell,\ell}$. So in this case, we reformulate our assumptions with the help of the equality $\widetilde{Q}_{\ell,m+1:\infty}= -\widetilde{Q}_{\ell,0:m}$.

\begin{rem} \label{rem:assumptionA1}
Assumption \ref{assum:second_index} essentially says that $Y$ decreases at a slower rate than $X$ moves from states to other states that belong to a lower class. Vice-versa, $Y$ increases at a faster rate than $X$ moves from states to other states that belong to a higher class. This can be formalized as follows.
\begin{itemize}
\item When $m < \ell$, the assumption is equivalent to:
\begin{equation*}
\widetilde{Q}_{\ell,0:m}
\leq \min_{i \in S_\ell} \, Q_{i,0:m}.
\end{equation*}
The quantity $\widetilde{Q}_{\ell,0:m}$ is the rate at which $Y$ moves from $\ell$ to classes that are smaller or equal to $m$ and the quantity $\min_{i \in S_\ell} \, Q_{i,0:m}$ is the minimal rate at which $X$ transitions from a state in class $\ell$ to another state in a class lower or equal to $m$.
\item When $m \geq \ell$, the assumption is equivalent to:
\begin{equation*}
\widetilde{Q}_{\ell,m+1:\infty}
\geq \max_{i \in S_\ell}\, Q_{i,m+1:\infty}.
\end{equation*}
The quantity $\widetilde{Q}_{\ell,m+1:\infty}$ is the rate at which $Y$ moves from $\ell$ to classes that are strictly greater than $m$ and the quantity $\max_{i \in S_\ell}\, Q_{i,m+1:\infty}$ is the maximal rate at which $X$ transitions from a state in class $\ell$ to another state in a class strictly higher than $m$.
\end{itemize}
Assumption \ref{assum:second_index} preserves the order $\cl{X} \leq Y$ when the two processes $X$ and $Y$ are in the same class.
\end{rem}

\begin{rem} \label{rem:assumptionA2}
Assumption \ref{assum:first_index} essentially says that $Y$ moves to low classes at a faster rate when the initial state is small, and vice-versa: Y moves to larger classes at a faster rate when the initial state is large. This can be formalized as follows.
\begin{itemize}
\item When $m< \ell$, the assumption is equivalent to:
\begin{equation*}
\widetilde{Q}_{\ell+1,0:m} \leq \widetilde{Q}_{\ell,0:m}.
\end{equation*}
The quantities $\widetilde{Q}_{\ell+1,0:m}$ and $\widetilde{Q}_{\ell,0:m}$ describes the rate at which $Y$ moves from class $\ell+1$, respectively from class $\ell$, to classes that are lower or equal to $m$.
\item When $m > \ell$, the assumption is equivalent to
\begin{equation*}
\widetilde{Q}_{\ell+1,m+1:\infty} \geq \widetilde{Q}_{\ell,m+1:\infty}.
\end{equation*}
The quantities $\widetilde{Q}_{\ell+1,m+1:\infty}$ and $\widetilde{Q}_{\ell,m+1:\infty}$ describe the rates at which $Y$ moves from class $\ell+1$, respectively from class $\ell$, to classes that are strictly greater than $m$.
\end{itemize}
Assumption \ref{assum:first_index} ensures that even when processes $X$ and $Y$ are not in the same class, the order $\cl{X} \leq Y$ is preserved. Indeed, $Y$ jumps multiple classes with a rate decreasing with respect to the number of skipped classes, preventing the process $Y$ to be smaller than the class of $X$ even when the two processes are far apart.
\end{rem}

\subsection{A preliminary result}
\label{subsec:transport}

We denote by $\seq{}$ the set of sequences $u=(u_k)_{k \in \mathbb{N}}$
of non-negative real numbers whose sum is finite. Following the convention introduced in the introduction, we denote $u_{0:\infty}= \sum_{k=0}^\infty u_k$ for a given $u \in \seq{}$.
Similarly, we define $\seqseq{}$ as the set of doubly-indexed sequences $\pi= (\pi_{k,\ell})_{(k,\ell) \in \mathbb{N}^2}$ of non-negative numbers whose sum is finite.
Given $\pi \in \seqseq{}$, $k \in \mathbb{N}$, and $\ell \in \mathbb{N}$, we denote $\pi_{k,0:\infty}= \sum\limits_{\ell=0}^\infty \pi_{k,\ell}$ and
$\pi_{0:\infty,\ell} = \sum\limits_{k=0}^\infty \pi_{k,\ell}$.

\begin{prop} \label{prop:transportation}
Let $a$ and $b$ be in $\seq{}$. Assume that $a_{0:k} \geq b_{0:k}$, for all $k \in \mathbb{N}$. Assume further that $a_{0:\infty}= b_{0:\infty}$. Then there exists $\pi \in \seqseq{}$ satisfying the following properties:
\begin{equation}
\label{eq:transport}
\pi_{k,0:\infty}= a_k, \quad
\pi_{0:\infty,\ell}= b_\ell, \quad
k> \ell \Rightarrow \pi_{k,\ell}= 0,
\quad
\forall (k,\ell) \in \mathbb{N}^2.
\end{equation}
\end{prop}

The rest of the subsection is dedicated to the construction of such a sequence $\pi$. One can understand Proposition \ref{prop:transportation} from the point of view of optimal transportation theory: a commodity is available in sources, labeled by $k= 0,1,\ldots$, and must be shipped to some factories, labeled by $\ell= 0,1,\ldots$. Then $a_k$ represents the amount of commodity available at source $k$, $b_{\ell}$ the capacity of factory $\ell$, and $\pi_{k,\ell}$ is the amount of commodity that is shipped from the source $k$ to the factory $\ell$. The first two conditions in \eqref{eq:transport} simply say that there is no loss of commodity. The third condition says that the commodity should be shipped from a given source to a factory with a greater or equal label.\\

In order to prove \Cref{prop:transportation}, we first consider the problem of shipping an amount $x$ of commodity (from a single source) to factories of capacities $u_0$, $u_1$,... favoring factories of lower indexes. So if $x \leq u_0$, we ship the whole amount to factory 0. Otherwise, we ship the amount $u_0$ to the factory 0. Then the remaining amount $x-u_0$ is entirely shipped to the factory 1 if $x-u_0 \leq u_1$. Otherwise, we ship the amount $u_1$, and so on. This leads to the following mathematical definition. Given $u \in \seq{}$ and $x \in [0,u_{0:\infty}]$, we define the sequence $(\Pi[x,u]_k)_{k \in \mathbb{N}}$ by
\begin{equation*}
\Pi[x,u]_k
=
\min(\max(x-u_{0:k-1},0) ,u_k),
\end{equation*}
with the convention that $u_{0:-1}=0$.
Two cases must be considered in order to interpret the formula. If $x \leq u_{0:k-1}$, then the total capacity of the factories $0,\ldots,k-1$ is larger than the amount of commodity, which does not need to be shipped to factory $k$. In this case we indeed have $\Pi[x,u]_k= \min(0,u_k)= 0$. Otherwise, when $x > u_{0:k-1}$, the factories $0,\ldots,k-1$ should be used at their full capacity. The above formula simplifies to $\Pi[x,u]_k= \min(x-u_{0:k-1},u_k)$.
The quantity $x-u_{0:k-1}>0$ represents the remaining amount to be allocated. If this quantity is less then $u_k$, it can be fully allocated to factory $k$, otherwise, we restrict it to $u_k$. 

We have the following technical result.

\begin{lem}\label{lem:easy_transport}
Let $u \in \seq{}$ and let $x \in [0, u_{0:\infty}]$. The five following statements hold true.
\begin{enumerate}
\item{\label{item3}} For all $k \in \mathbb{N}$,~ $\Pi[x,u]_k \leq u_k$.
\item{\label{item4}}  For any $k \in \mathbb{N}$, if $u_{0:k} \leq x$, then $\Pi[x,u]_k = u_k$.
\item{\label{item2}} If $x= u_{0:\infty}$, then $\Pi[x,u]=u$.
\item{\label{item1}} We have $\Pi[x,u] \in \seq{}$ and in particular, $\Pi[x,u]_{0:\infty}= x$.
\item{\label{item5}}  For any $x' \in [0,x]$, for all $k \in \mathbb{N}$, we have
 $$0 \leq \Pi[x,u]_k - \Pi[x',u]_k \leq x-x'.$$
\end{enumerate}
\end{lem}

\begin{proof}
For convenience, we set $v= \Pi[x,u]$.
\begin{enumerate}
\item Straightforward.
\item Let $k \in \mathbb{N}$ be such that $u_{0:k} \leq x$. Then we have $x-u_{0:k-1} = x-u_{0:k} + u_k \geq u_k \geq 0$. Therefore $\Pi[x,u]_k= \min(x-u_{0:k-1},u_k)= u_k$ as was to be shown.
\item If $x= u_{0:\infty}$, then we have $u_{0:k} \leq x$ for all $k$, and therefore $\Pi[x,u]_k= u_k$ (by the previous point).
\item From the definition of $v$, we have $v_k \geq 0$. We only have to prove that $v_{0:\infty}=x$. If $x= u_{0:\infty}$, we have $v= u$ and the claim holds true. Otherwise, if $x < u_{0:\infty}$, there exists $K \in \mathbb{N}$ such that
\begin{equation*}
u_{0:K-1} < x \leq u_{0:K}.
\end{equation*}
It is easy to verify that for $k=0,\ldots, K-1$, $v_k= u_k$, that for $k= K$, $v_K= x-u_{0:K-1}$ and that for $k>K$, $v_k=0$. The claim follows.
\item Let $k \in \mathbb{N}$. Then $\Pi[x,u]_k$ results from the composition of two functions, that is, $\Pi[x,u]_k= \alpha_k \circ \beta_k(x)$, where
\begin{equation*}
\beta_k(x)= \max(x-u_{0:k-1},0)
\quad \text{and} \quad
\alpha_k(y)= \min(y,u_k).
\end{equation*}
Both functions $\alpha_k$ and $\beta_k$ are non-decreasing and Lipschitz continuous with modulus 1. Therefore, the composition $\alpha_k \circ \beta_k$ is also non-decreasing and Lipschitz continuous with modulus 1. The claim immediately follows.
\end{enumerate}
\end{proof}

Given $a$ and $b$ in $\seq{}$ satisfying the assumptions of Proposition \ref{prop:transportation}, we define $\bar{\Pi}[a,b] \colon (k,\ell) \in \mathbb{N}^2 \rightarrow \bar{\Pi}[a,b]_{k,\ell} \in \mathbb{R}$ as
\begin{equation*}
\bar{\Pi}[a,b]_{k,\ell}
= \Pi[a_{0:k},b]_\ell
- \Pi[a_{0:k-1},b]_{\ell}, \quad \forall (k,\ell) \in \mathbb{N}^2,
\end{equation*}
using again the convention $a_{0:-1}= 0$.
Since the sequences $a$ and $b$ satisfy the assumptions of Proposition \ref{prop:transportation}, for any $k \in \mathbb{N}$, $a_{0:k} \leq b_{0:\infty}$. Thus for any $k$, the two terms $\Pi[a_{0:k},b]$ and $\Pi[a_{0:k-1},b]$ are well-defined.

\begin{lem}
Let $a$ and $b$ satisfy the assumptions of Proposition \ref{prop:transportation}. Then $\bar{\Pi}[a,b]$ lies in $\seqseq{}$ and satisfies \eqref{eq:transport}.
\end{lem}

\begin{proof}
Let us define $\Lambda \colon (k,\ell) \in (\{ -1 \} \cup \mathbb{N}) \times \mathbb{N}
\mapsto \Lambda_{k,\ell} \in \mathbb{R}$ as
\begin{equation*}
\Lambda_{k,\ell}= \Pi[a_{0:k},b]_{\ell}.
\end{equation*}
For convenience, we set $\pi= \bar{\Pi}[a,b]$. We have $\pi_{k,\ell}= \Lambda_{k,\ell}- \Lambda_{k-1,\ell}$ for all $(k,\ell) \in \mathbb{N}^2$. First we verify that $\pi \in \seqseq{}$. As $a_{0:k} \geq a_{0:k-1}$, for any $k$, we have that $\Lambda_{k,\ell} \geq \Lambda_{k-1,\ell}$, by Lemma \ref{lem:easy_transport}.\ref{item5}. Thus $\pi_{k,\ell} \geq 0$.
Note that if $\pi$ satisfies \eqref{eq:transport} then $\pi$ has finite sum, so it only remains to verify that \eqref{eq:transport} holds true.
\begin{itemize}
\item Let $k \in \mathbb{N}$. Using Lemma \ref{lem:easy_transport}.\ref{item1}, we have that for any $\ell \in \mathbb{N}$,
\begin{equation*}
\pi_{k,0:\ell}=
\Lambda_{k,0:\ell}- \Lambda_{k-1,0:\ell} \underset{\ell \to+\infty}{\longrightarrow} a_{0:k}- a_{0:k-1}
= a_k.
\end{equation*}
\item For $k \in \mathbb{N}$,
\begin{equation*}
\pi_{0:k,\ell}
= \Lambda_{k,\ell}- \Lambda_{-1,\ell}
= \Lambda_{k,\ell}
= \Pi[a_{0:k},b]_{\ell}.
\end{equation*}
Using Lemma \ref{lem:easy_transport}.\ref{item5}, the function $x \mapsto \Pi[x,b]_l$ is $1$-Lipschitz so it is continuous. Therefore,
\begin{equation*}
\lim_{k \to +\infty} \pi_{0:k,\ell} = \Pi[a_{0:\infty},b]_{\ell}
= \Pi[b_{0:\infty},b]_\ell
= b_\ell.
\end{equation*}
The last equality follows from Lemma \ref{lem:easy_transport}.\ref{item2}.
\item Let $k > \ell$.
We have
\begin{equation*}
b_{0:\ell} \leq a_{0:\ell}
\leq a_{0:k-1} \leq a_{0:k}.
\end{equation*}
Therefore, by Lemma \ref{lem:easy_transport}.\ref{item4}, we have $\Lambda_{k,\ell}= b_\ell$ and $\Lambda_{k-1,\ell}= b_\ell$ and finally $\pi_{k,\ell}= b_{\ell}- b_{\ell}= 0$.
\end{itemize}
\end{proof}

\subsection{Proof of the main result}
\label{subsec:proof}

We prove in this subsection Theorem \ref{thm:main_upper-bound}. We assume that a transition-rate matrix $\widetilde{Q}$ satisfying Assumptions \ref{assum:second_index} and \ref{assum:first_index} is given (its construction will be investigated in Section \ref{subsec:construction_Qtilde}).

\begin{lem} \label{lem:cond:R}
Assume that there exists a transition-rate matrix $R= (R_{(i,\ell),(j,m)})_{(i,\ell),(j,m) \in \mathcal{S} \times \mathbb{N}}$ satisfying the following properties:
\begin{equation}
\label{eq:condition_R}
\begin{array}{rll}
a) & R_{(i,\ell),(j,0:\infty)} = Q_{i,j},
& \forall i \in \mathcal{S},\ \forall \ell \in \mathbb{N},\ \forall j \in \mathcal{S}, \\
b) & R_{(i,\ell),(0:\infty,m)} = \widetilde{Q}_{\ell,m},
& \forall i \in \mathcal{S}, \ \forall \ell \in \mathbb{N}, \ \forall m \in \mathbb{N}, \\
c) & \cl{i} \leq \ell \text{ and }
\cl{j} > m \Rightarrow R_{(i,\ell),(j,m)} = 0, & \forall (i,\ell) \in \mathcal{S} \times \mathbb{N}, \ \forall (j,m) \in \mathcal{S} \times \mathbb{N}.
\end{array}
\end{equation}
Then for any $\left( X_0, Y_0 \right) \in \mathcal{S} \times \mathbb{N}$ such that $\cl{X_0} \leq Y_0$, the continuous-time Markov chain $Y$ with transition-rate matrix $\tilde{Q}$ starting in $Y_0$ is an upper-bounding process of $X$.
\end{lem}

\begin{proof}
The transition-rate matrix allows us to define a coupling between the processes $X$ and $Y$.\\
The conditions (\ref{eq:condition_R}.$a$) and (\ref{eq:condition_R}.$b$) ensure that the marginals of this coupling are well-defined and have respectively the same law than the processes $X$ and $Y$.\\
Moreover the condition (\ref{eq:condition_R}.$c$) ensures that the jumps to state $(j,m) \in \mathcal{S} \times \mathbb{N}$ such that $\cl{j} > m$ are not possible.\\
Thus since the initial condition verifies $\cl{X_0} \leq Y_0$, the order $\cl{X} \leq Y$ is preserved over time.
\end{proof}

The proof of Theorem \ref{thm:main_upper-bound} amounts to finding a transition matrix $R$ satisfying \eqref{eq:condition_R}.
The search for a transition matrix $R$ will be successively simplified in the following lemmas, so as to boil down to an application of Proposition \ref{prop:transportation}.

The first transformation of the problem that we realize simply consists of a translation of $R$, $Q$ and $\tilde{Q}$. For any pair $(i,\ell) \in \mathcal{S} \times \mathbb{N}$, we define $M_{i,\ell}$ as
\begin{equation*}
M_{i,\ell} = - Q_{i,i} - \widetilde{Q}_{\ell,\ell} \geq 0.
\end{equation*}
For any pair $(i,j) \in \mathcal{S}^2$ and for any pair $(\ell,m) \in \mathbb{N}^2$, we define $Q_{(i,\ell),j}'$ and $\widetilde{Q}_{(i,\ell),m}'$ as
\begin{equation*}
\begin{array}{ll}
Q_{(i,\ell),j}'=
Q_{i,j} + M_{i,\ell} ~\delta_{i=j}, & \forall (i,j) \in \mathcal{S}^2, \\
\widetilde{Q}_{(i,\ell),m}'=
\widetilde{Q}_{\ell,m} + M_{i,\ell} ~\delta_{\ell=m}, & \forall (\ell,m) \in \mathbb{N}^2.
\end{array}
\end{equation*}
We use here the Kronecker symbol, i.e., $\delta_{i=j}=1$ if $i=j$, otherwise $\delta_{i=j}=0$. We directly see from the above definitions that $Q_{(i,\ell),j}' \geq 0$ and $\tilde{Q}_{(i,\ell),m}' \geq 0$.

\begin{lem} \label{lem:cond:R'}
Let $R'= (R_{(i,\ell),(j,m)}')_{(i,\ell),(j,m) \in \mathcal{S} \times \mathbb{N}}$ satisfy the following properties:
\begin{equation}\label{eq:conditions_R'}
\begin{array}{rll}
a) & R_{(i,\ell),(j,0:\infty)}' = Q_{(i,\ell),j}' ~, & \forall i \in \mathcal{S},\ \forall \ell \in \mathbb{N},\ \forall j \in \mathcal{S}, \\
b) & R_{(i,\ell),(0:\infty,m)}' = \widetilde{Q}_{(i,\ell),m}' ~, & \forall i \in \mathcal{S}, \ \forall \ell \in \mathbb{N}, \ \forall m \in \mathbb{N}, \\
c) & \cl{i} \leq \ell \text{ and } \cl{j} > m \Rightarrow R_{(i,\ell),(j,m)}' = 0  ~, \ \ & \forall (i,\ell) \in \mathcal{S} \times \mathbb{N}, \ \forall (j,m) \in \mathcal{S} \times \mathbb{N}, \\
d) & R_{(i,\ell),(j,m)}' \geq 0 ~, & \forall (i,\ell) \in \mathcal{S} \times \mathbb{N}, \ \forall (j,m) \in \mathcal{S} \times \mathbb{N}.
\end{array}
\end{equation}
Then the matrix $R= (R_{(i,\ell),(j,m)})_{(i,\ell),(j,m) \in \mathcal{S} \times \mathbb{N}}$ defined by
\begin{equation*}
R_{(i,\ell),(j,m)}=
R_{(i,\ell),(j,m)}' - M_{i,\ell} \delta_{(i,\ell)=(j,m)},
\end{equation*}
is a transition-rate matrix satisfying \eqref{eq:condition_R}.
\end{lem}

\begin{proof}
Let us verify that $R$ is a transition-rate matrix. The condition (\ref{eq:conditions_R'}.$d$) and the definition of $R$ imply the nonnegativity of the non-diagonal terms of $R$. Moreover, we have
\begin{equation*}
\sum_{(j,m) \in \mathcal{S} \times \mathbb{N}} R_{(i,\ell),(j,m)}= \sum_{(j,m) \in \mathcal{S} \times \mathbb{N}} R_{(i,\ell),(j,m)}'-M_{i,\ell}
= \sum_{j \in \mathcal{S}} Q_{(i,\ell),j}' - M_{i,\ell}
= \sum_{j \in \mathcal{S}} Q_{i,j} + M_{i,\ell}- M_{i,\ell} = 0,
\end{equation*}
and
\begin{equation*}
\sum_{\substack{(j,m) \in \mathcal{S} \times \mathbb{N}\\ (j,m) \ne (i,\ell)}} R_{(i,\ell),(j,m)}= - R_{(i,\ell),(i,\ell)}'+M_{i,\ell} < \infty,
\end{equation*}
which proves that $R$ is a transition-rate matrix.

Next, using condition (\ref{eq:conditions_R'}.$a$) and the definition of $Q'$, we have that
\begin{equation*}
\sum_{m \in \mathbb{N}} R_{(i,\ell),(j,m)} =
\sum_{m \in \mathbb{N}} R_{(i,\ell),(j,m)}'- M_{i,\ell} \delta_{i=j}
=  Q_{(i,\ell),j}' - M_{i,\ell} \delta_{i=j}
= Q_{i,j},
\end{equation*}
which implies condition (\ref{eq:condition_R}.$a$).
Condition (\ref{eq:condition_R}.$b$) can be verified in a similar fashion.
Finally, consider $i$, $j$, $\ell$, and $m$ such that $\cl{i} \leq \ell$ and $\cl{j} > m$. Then $(i,\ell) \neq (j,m)$, which implies that
$R_{(i,\ell),(j,m)} = R'_{(i,\ell),(j,m)}= 0$, which proves (\ref{eq:condition_R}.$c$) and concludes the proof.
\end{proof}

The following lemma shows that a matrix $R'$ satisfying \eqref{eq:conditions_R'} can be deduced from a matrix $R''= (R_{(i,\ell),(k,m)}'')_{(i,\ell) \in \mathcal{S}\times \mathbb{N},\, (k,m) \in \mathbb{N}^2}$ that is of simpler form, in so far as the third index of $R''$ lies in $\mathbb{N}$ (while the third index of $R'$ lies in $\mathcal{S}$).
Note that in the statement of condition (\ref{eq:conditions_R''}.$a$) below, we make use of the convention \eqref{eq:convention1}, so we have $Q_{(i,\ell),k}'= \sum_{j \in S_k} Q_{(i,\ell),j}$.\\

\begin{lem} \label{lem:cond:R''}
Let $R''= (R_{(i,\ell)(k,m)}'')_{(i,\ell) \in \mathcal{S} \times \mathbb{N}, (k,m) \in \mathbb{N}^2}$ be such that
\begin{equation}
\label{eq:conditions_R''}
\begin{array}{rll}
a) & R_{(i,\ell),(k,0:\infty)}'' = Q_{(i,\ell),k}' ~, & \forall i \in \mathcal{S},\ \forall \ell \in \mathbb{N},\ \forall k \in \mathbb{N}, \\
b) & R_{(i,\ell),(0:\infty,m)}'' = \widetilde{Q}_{(i,\ell),m}' ~,
& \forall i \in \mathcal{S}, \ \forall \ell \in \mathbb{N}, \ \forall m \in \mathbb{N}, \\
c) & \cl{i} \leq \ell \text{ and }
k > m \Rightarrow R_{(i,\ell),(k,m)}'' = 0 ~, \ \ & \forall (i,\ell) \in \mathcal{S} \times \mathbb{N}, \ \forall (k,m) \in \mathbb{N}^2, \\
d) & R_{(i,\ell),(k,m)}'' \geq 0 ~, & \forall (i,\ell) \in \mathcal{S} \times \mathbb{N},\ \forall (k,m) \in \mathbb{N}^2.
\end{array}
\end{equation}
Then, the matrix $R'$ defined
for all $((i,\ell),(j,m)) \in (\mathcal{S}\times \mathbb{N})^2$ by
\begin{equation*}
R_{(i,\ell),(j,m)}'
=
\begin{cases}
\begin{array}{cl}
\frac{Q_{(i,\ell),j}'}{Q_{(i,\ell),\cl{j}}'}
R''_{(i,\ell),(\cl{j},m)} &
\text{ if $Q_{(i,\ell),\cl{j}}' >0$}, \\[1em]
0 & \text{otherwise},
\end{array}
\end{cases}
\end{equation*}
satisfies \eqref{eq:conditions_R'}.
\end{lem}

\begin{proof}
Let us verify (\ref{eq:conditions_R'}.$a$).
If $Q_{(i,\ell),\cl{j}}' > 0$, we have with the definition of $R'$ and condition (\ref{eq:conditions_R''}.$a$) that
\begin{equation*}
\sum_{m \in \mathbb{N}} R_{(i,\ell),(j,m)}'
= \frac{Q_{(i,\ell),j}'}{Q_{(i,\ell),\cl{j}}'}
\sum_{m \in \mathbb{N}}
R''_{(i,\ell),(\cl{j},m)}
= \frac{Q_{i,j}'}{Q_{i,\cl{j}}'} Q_{(i,\ell),\cl{j}}'
= Q_{(i,\ell),j}',
\end{equation*}
as was to be proved. Otherwise, if $Q_{(i,\ell),\cl{j}}' = 0$, then necessarily $Q_{(i,\ell),j}'=0$, and by definition, $R_{(i,\ell),(j,m)}'=0$, which shows that (\ref{eq:conditions_R'}.$a$) also holds in that case.

Let us verify (\ref{eq:conditions_R'}.$b$). We have
\begin{equation*}
R_{(i,\ell),(0:\infty,m)}' = \sum_{k \in \mathbb{N}}
\sum_{j \in S_k}  R_{(i,\ell),(j,m)}'.
\end{equation*}
If $Q_{(i,\ell),k}' \neq 0$, we have
\begin{equation*}
\sum_{j \in S_k}  R_{(i,\ell),(j,m)}'
= \sum_{j \in S_k}
\Bigg( \frac{Q_{(i,\ell),j}'}{Q_{(i,\ell),\cl{j}}'}
R''_{(i,\ell),(\cl{j},m)} \Bigg)
= \frac{R''_{(i,\ell),(k,m)}}{Q_{(i,\ell),k}'} \sum_{j \in S_k} Q_{(i,\ell),j}'
= R_{(i,\ell),(k,m)}''.
\end{equation*}
If $Q_{(i,\ell),k}'=0$, then
$\sum_{j \in S_k}  R_{(i,\ell),(j,m)}'=0$ and $R_{(i,\ell),(k,m)}''=0$ (by conditions $a$ and $d$). Thus again, we have
\begin{equation*}
\sum_{j \in S_k}  R_{(i,\ell),(j,m)}'
= R_{(i,\ell),(k,m)}''.
\end{equation*}
Combining the three above equalities and (\ref{eq:conditions_R''}.$b$), we deduce that
\begin{equation*}
R_{(i,\ell),(0:\infty,m)}' = \sum_{k \in \mathbb{N}}
R_{(i,\ell),(k,m)}''
= \widetilde{Q}_{(i,\ell),m}',
\end{equation*}
as was to be shown.

Let $i$, $\ell$, $j$, and $m$ be such that $\cl{i} \leq \ell$ and $\cl{j} > m$. Then, by condition (\ref{eq:conditions_R''}.$c$), $R_{(i,\ell), (\cl{j},m)}''=0$ and thus
$R'_{(i,\ell),(j,m)} = 0$,
which proves that (\ref{eq:conditions_R'}.$c$) is satisfied.
Condition (\ref{eq:conditions_R'}.$d$) is also satisfied, as a direct consequence of (\ref{eq:conditions_R''}.$d$).
\end{proof}

\begin{lem} \label{lem:existence:R''}
Let Assumptions \ref{assum:second_index} and \ref{assum:first_index} hold true. Then there exists $R''$ satisfying \eqref{eq:conditions_R''}.
\end{lem}

\begin{proof}
The construction of $R''$ is realized for each pair $(i,\ell)$, separately from one another. So we fix $(i,\ell)$ and consider two cases.
\begin{itemize}
\item Case 1: $\cl{i} \leq \ell$. The construction of $R_{(i,\ell),(\cdot,\cdot)}''$ follows from Proposition \ref{prop:transportation}, applied with the sequences $a:= \big(Q_{(i,\ell),j}' \big)_{j \in \mathbb{N}}$ and $b:= \big(\widetilde{Q}_{(i,\ell),m}' \big)_{m \in \mathbb{N}}$. The result of Proposition \ref{prop:transportation} indeed implies all statements of \eqref{eq:conditions_R''}. It remains to verify the applicability of the proposition. First we note that $a$ and $b$ are nonnegative sequences. Moreover,
\begin{equation*}
a_{0:\infty}
= Q_{(i,\ell),0:\infty}'
= Q_{i,0:\infty} + M_{i,\ell}
= M_{i,\ell},
\end{equation*}
and similarly, $b_{0:\infty}= M_{i,\ell}$; so $a, b \in \seq{}$. It remains to verify that for any $k \in \mathbb{N}$, $a_{0:k} \geq b_{0:k}$. Let $k \in \mathbb{N}$. Note first that
\begin{align*}
a_{0:k}
= Q_{(i,\ell),0:k}'
= Q_{i,0:k} + M_{i,\ell} ~\delta_{k \geq \cl{i}}, \\
b_{0:k}
= \widetilde{Q}_{(i,\ell),0:k}'
= \widetilde{Q}_{\ell,0:k} + M_{i,\ell} ~\delta_{k \geq \ell}.
\end{align*}
Next we consider two cases.
\begin{itemize}
\item Case (a): $k \notin \{ \cl{i},\ldots,\ell-1 \}$. Using Assumption \ref{assum:second_index} once and Assumption \ref{assum:first_index} repeatedly, we obtain that
\begin{equation*}
Q_{i,0:k}
\geq \widetilde{Q}_{\cl{i},0:k}
\geq \widetilde{Q}_{\cl{i}+1,0:k}
\geq \ldots
\geq \widetilde{Q}_{\ell,0:k}.
\end{equation*}
Next, as $\cl{i} \leq \ell$, we have
\begin{equation*}
\delta_{k \geq \cl{i}} \geq \delta_{k \geq \ell}.
\end{equation*}
Combining the two derived inequalities with the expressions of $a_{0:k}$ and $b_{0:k}$, we deduce that $a_{0:k} \geq b_{0:k}$ as was to be shown.
\item Case (b): $k \in \{ \cl{i},\ldots,\ell-1 \}$. Note first that the above proof does not work anymore, as it would require us to utilize the inequality $\widetilde{Q}_{k,0:k} \geq \widetilde{Q}_{k+1,0:k}$, which we do not assume because it is wrong in general (indeed: $\widetilde{Q}_{k,0:k}\leq 0$ and $\widetilde{Q}_{k+1,0:k} \geq 0$). So we proceed differently. We have the inequalities:
\begin{equation*}
Q_{i,0:k} \geq Q_{i,i} \quad
\text{and} \quad
\widetilde{Q}_{\ell,0:k} \leq -\widetilde{Q}_{\ell,\ell} ~,
\end{equation*}
which are actually true for any $k$ in $\mathbb{N}$. As we consider here the case with $k \geq \cl{i}$ and $k<\ell$, we have:
\begin{align*}
a_{0:k} \geq {} & Q_{i,i} + M_{i,\ell} ~\delta_{k \geq \cl{i}}
= Q_{i,i} + M_{i,\ell}
= -\widetilde{Q}_{\ell,\ell} ~, \\
b_{0:k} \leq {} & - \widetilde{Q}_{\ell,\ell} + M_{i,\ell} ~ \delta_{k \geq \ell} = -\widetilde{Q}_{\ell,\ell}.
\end{align*}
So $a_{0:k} \geq b_{0:k}$, as was to be shown.
\end{itemize}
\item Case 2: $\cl{i} > \ell$. This case is easier to treat, as condition (\ref{eq:conditions_R''}.$c$) is trivially satisfied. Define
\begin{equation*}
R_{(i,\ell)(k,m)}''
=
\begin{cases}
\begin{array}{cl}
\frac{1}{M_{i,\ell}} Q_{(i,\ell),k} \widetilde{Q}_{(i,\ell),m}, &
\text{ if } M_{i,\ell} \geq 0, \\
0, & \text { otherwise. }
\end{array}
\end{cases}
\end{equation*}
The nonnegativity of $R_{(i,\ell)(k,m)}''$ is clear in both cases. If $M_{i,\ell} > 0$, then
\begin{equation*}
R_{(i,\ell),(k,0:\infty)}
= \frac{1}{M_{i,\ell}} Q_{(i,\ell),k} Q_{(i,\ell),0:\infty}'
= \frac{1}{M_{i,\ell}} Q_{(i,\ell),k} M_{i,\ell} = Q_{(i,\ell),k},
\end{equation*}
so condition (\ref{eq:conditions_R''}.$a$) holds true. If $M_{i,\ell}=0$, then $Q_{i,i}=0$ and $Q_{i,j}=0$, for any $j$. This further implies that $Q_{(i,\ell),k}'=0$, and so condition (\ref{eq:conditions_R''}.$a$) also holds true in that case. Condition (\ref{eq:conditions_R''}.$b$) can be verified with similar arguments.
\end{itemize}
\end{proof}

\begin{proof}[of Theorem \ref{thm:main_upper-bound}]
Lemma \ref{lem:existence:R''} gives the existence of a matrix $R''$ satisfying the assumptions of Lemma \ref{lem:cond:R''}, which in turn gives the existence of a matrix $R'$ satisfying the assumptions of Lemma \ref{lem:cond:R'}, which gives the existence of a matrix $R$ satisfying the assumptions of Lemma \ref{lem:cond:R}, which finally implies Theorem \ref{thm:main_upper-bound}.
\end{proof}

\subsection{Construction of an upper-bounding transition-rate matrix}
\label{subsec:construction_Qtilde}

We address the issue of finding a transition-rate matrix $\widetilde{Q}$ that satisfies Assumptions \ref{assum:second_index} and \ref{assum:first_index}. Our analysis relies on a transformation of the transition-rate matrices $\widetilde{Q} \in \trmat{\mathbb{N}}$, which we define first. Consider the sets:
\begin{equation*}
\Omega^-
= \Big\{ (\ell,m) \in \mathbb{N}^2 \,|\, m< \ell \Big\}
\quad \text{and} \quad
\Omega^+
= \Big\{ (\ell,m) \in \mathbb{N}^2 \,|\, m > \ell \Big\}.
\end{equation*}
Consider the map $\Phi \colon \widetilde{Q} \in \trmat{\mathbb{N}} \mapsto \Phi(\widetilde{Q})$, where $\Phi(\widetilde{Q})= (U^-,U^+)$ and where
\begin{equation*}
\begin{array}{l}
U^- = \big( \widetilde{Q}_{\ell,0:m} \big)_{(\ell,m) \in \Omega^-}, \\
U^+ = \big( \widetilde{Q}_{\ell,m:\infty} \big)_{(\ell,m) \in \Omega^+}.
\end{array}
\end{equation*}
We introduce next the set $\mathbf{U}$, defined as the set of pairs $(U^-,U^+)$ where $U^-= \big( U_{\ell,m}^- \big)_{(\ell,m) \in \Omega^-}$ and $U^+ = \big( U_{\ell,m}^+ \big)_{(\ell,m) \in \Omega^+}$ satisfy the following:
\begin{equation} \label{eq:cond:Umono}
\begin{array}{rll}
a) & U_{\ell,m-1}^- \leq U_{\ell,m}^- & \forall (\ell,m) \in \Omega^- \text{ such that } m>0, \\
b) & U_{\ell,m}^+ \geq U_{\ell,m+1}^+,
& \forall (\ell,m) \in \Omega^+,
\end{array}
\end{equation}
and
\begin{equation} \label{eq:cond:Uboundary}
\begin{array}{rl}
a) & U^-_{\ell,0} \geq 0, \quad \forall \ell \in \mathbb{N}^*, \\
b) & \underset{m \to+\infty}{\lim} U_{\ell,m}^+ = 0, \quad \forall \ell \in \mathbb{N}.
\end{array}
\end{equation}

\begin{lem}\label{lem:bijec}
The mapping $\Phi$ is a bijection from $\trmat{\mathbb{N}}$ to $\mathbf{U}$. Its inverse is given by
\begin{equation*}
\Phi^{-1}(U^-,U^+)_{\ell,m}
= \begin{cases}
\begin{array}{cl}
U^-_{\ell,m} - U^-_{\ell,m-1} & \text{ if $(\ell,m) \in \Omega^- $}, \\
- U^-_{\ell,\ell-1} - U^+_{\ell,\ell+1} & \text{ if $\ell=m$}, \\
U^+_{\ell,m} - U^+_{\ell,m+1} & \text{ if $(\ell, m) \in \Omega^+$}.
\end{array}
\end{cases}
\end{equation*}
with the convention for all $\ell \in \mathbb{N}$, $U^-_{\ell,-1} =0$.
\end{lem}

\begin{proof}
Let us call $\Psi$ the function from $\mathbf{U}$ to $\trmat{\mathbb{N}}$ defined in the lemma. As a first step, we will prove that the mappings $\Phi$ and $\Psi$ are well-defined.\\
Let $Q \in \trmat{\mathbb{N}}$ and $(U^-,U^+) = \Phi(Q)$. We need to check that $(U^-,U^+)$ is in $\mathbf{U}$, that is it verifies \eqref{eq:cond:Umono} and \eqref{eq:cond:Uboundary}.\\
Since for any $\ell \ne m \in \mathbb{N}$, $Q_{\ell,m} \geq 0$, by definition of $U_{\ell,m}^-$ and $U_{\ell,m}^+$, the equation \eqref{eq:cond:Umono} is verified by $(U^-,U^+)$.\\
Moreover, for all $\ell \geq 1$, $U^-_{\ell,0} = Q_{\ell,0:0} = Q_{\ell,0} \geq 0$. Finally, for any $\ell \in \mathbb{N}$, by definition of $\trmat{\mathbb{N}}$, the sum of the sequence $\big( Q_{\ell,m} \big)_{m > \ell}$ is finite, so the limit of the rest of this sum goes to $0$ and we have proved the condition \eqref{eq:cond:Uboundary}b.\\
Let $(U^-,U^+) \in \mathbf{U}$ and $Q = \Psi(U^-,U^+)$. We need to check that $Q \in \trmat{\mathbb{N}}$.\\
With \eqref{eq:cond:Umono} and \eqref{eq:cond:Uboundary}a, the positivity of $Q_{\ell,m}$ for any $\ell \ne m$ is straightforward. Let us compute the sum according to the second index, we have on one hand for any $(\ell,m) \in \Omega^-$,
$$ Q_{\ell,0:\ell-1} = U^-_{\ell,0} + \sum_{k = 1}^{\ell -1} \big( U^-_{\ell,k} - U^-_{\ell,k-1} \big) = U^-_{\ell,\ell-1},$$
and on the other hand, with \eqref{eq:cond:Uboundary}b, for any $(\ell,m) \in \Omega^+$:
$$ \sum_{k = \ell+1}^m Q_{\ell,k} = \sum_{k = \ell+1}^m \big( U^+_{\ell,k} - U^+_{\ell,k+1} \big) = U^+_{\ell, \ell +1} - U^+_{\ell,m+1} \underset{m \to +\infty}{\longrightarrow} U^+_{\ell, \ell +1}. $$
So $\sum\limits_{m \ne \ell} Q_{\ell,m} = U^-_{\ell,\ell-1} + U^+_{\ell, \ell +1} < \infty$ and $\sum\limits_{m \in \mathbb{N}} Q_{\ell,m} = 0$.\\
It only remains to prove that for any $Q \in \trmat{\mathbb{N}}$, $\Psi \big( \Phi(Q) \big) = Q$. For convenience, we set $V_{\ell,m} = \Psi \big( \Phi(Q) \big)_{\ell,m}$ for all $(\ell,m) \in \mathbb{N}^2$ and $\Phi(Q) = (U^-,U^+)$.\\
Let $(\ell,m) \in \mathbb{N}^2$, three cases must be considered.
\begin{itemize}
\item If $(\ell, m) \in \Omega^-$, $V_{\ell,m} = U^-_{\ell,m}- U^-_{\ell,m-1} = Q_{\ell,0:m} - Q_{\ell,0:m-1} =  Q_{\ell,m}$.
\item If $m=\ell$, $V_{\ell,\ell} = -U^-_{\ell,\ell-1}- U^+_{\ell,\ell+1} = -Q_{\ell,0:\ell-1} - Q_{\ell,\ell+1:\infty} =  Q_{\ell,\ell}$ where the last equality comes from the fact that the sum according to the second coordinate of $Q$ is equal to $0$.
\item If $(\ell,m) \in \Omega^+$, $V_{\ell,m} = U^+_{\ell,m}- U^+_{\ell,m+1} = Q_{\ell,m:\infty} - Q_{\ell,m+1:\infty} =  Q_{\ell,m}$.
\end{itemize}
This concludes the proof.
\end{proof}

Next we define $f^-= (f_{\ell,m}^-)_{(\ell,m) \in \Omega^- }$ and $f^+= (f_{\ell,m}^+)_{(\ell,m) \in \Omega^+ }$ as
\begin{align}
f_{\ell,m}^-= {} \min_{i \in S_{\ell}} \, Q_{i,0:m} \quad \text{and} \quad
f_{\ell,m}^+ = {} \max_{i \in S_{\ell}} \, Q_{i,m:\infty}, \label{eq:def_f_upper}
\end{align}
which are well-defined since each class is finite.\\

\begin{lem}\label{lem:2.8}
Let $\widetilde{Q} \in \trmat{\mathbb{N}}$ and let $(U^-,U^+)= \Phi(\widetilde{Q})$.
Then $\widetilde{Q}$ satisfies Assumption \ref{assum:second_index} if and only if
\begin{equation} \label{eq:condUA1}
\begin{array}{rll}
a) & U_{\ell,m}^- \leq f_{\ell,m}^-, & \forall (\ell,m) \in \Omega^-, \\
b) & U_{\ell,m}^+ \geq f_{\ell,m}^+, & \forall (\ell,m) \in \Omega^+.
\end{array}
\end{equation}
Moreover, $\widetilde{Q}$ satisfies Assumption \ref{assum:first_index} if and only if
\begin{equation} \label{eq:condUA2}
\begin{array}{rll}
a) & U_{\ell+1,m}^- \leq U_{\ell,m}^-
& \forall (\ell,m) \in \Omega^-, \\
b) & U_{\ell+1,m+1}^+ \geq U_{\ell,m+1}^+, & \forall (\ell,m) \in \Omega^+.
\end{array}
\end{equation}
\end{lem}

\begin{proof}
The result directly follows from the case disjunction realized in Remarks \ref{rem:assumptionA1} and \ref{rem:assumptionA2}.
\end{proof}

A consequence of the three above lemmas is that the problem of finding a transition-rate matrix $\widetilde{Q}$ satisfying Assumptions \ref{assum:second_index} and \ref{assum:first_index} is equivalent to finding $U^-$ and $U^+$ satisfying \eqref{eq:cond:Umono}, \eqref{eq:cond:Uboundary}, \eqref{eq:condUA1}, and \eqref{eq:condUA2}. Note that the conditions to be satisfied by $U^-$ are independent from those to be satisfied by $U^+$. We define next two candidates $\bar{U}^-$ and $\bar{U}^+$ as follows:

\begin{equation}\label{eq:optimal_upper}
\begin{array}{rll}
\bar{U}^-_{\ell,m}
= & \min\limits_{\ell' \in \{ m+1, \ldots, \ell \} } f_{\ell',m }^- & \forall (\ell,m) \in \Omega^-,\\
\bar{U}^+_{\ell,m}
= & \max\limits_{\ell' \in \{ 0, \ldots, \ell \} } f_{\ell',m }^+ & \forall (\ell,m) \in \Omega^+.
\end{array}
\end{equation}

\begin{prop}\label{prop:optimal_upper}
Firstly, the quantities $\bar{U}^-$ and $\bar{U}^+$ are well-defined and satisfy \eqref{eq:cond:Umono}, \eqref{eq:cond:Uboundary}, \eqref{eq:condUA1}, and \eqref{eq:condUA2}.
Therefore $\Phi^{-1}(\bar{U}^-,\bar{U}^+)$ is a transition-rate matrix satisfying Assumptions \ref{assum:second_index} and \ref{assum:first_index}.\\
Moreover, the transition-rate matrix $\Phi^{-1}(\bar{U}^-,\bar{U}^+)$ is optimal in the following sense: any other transition-rate matrix $\widetilde{Q}$ satisfying Assumptions \ref{assum:second_index} and \ref{assum:first_index} is necessarily such that
\begin{equation*}
\widetilde{Q}_{\ell,0:m}
\leq \Phi^{-1}(\bar{U}^-,\bar{U}^+)_{\ell,0:m}, \quad
\forall (\ell,m) \in \mathbb{N}^2.
\end{equation*}
\end{prop}

\begin{proof}
Since $\bar{U}^-$ (respectively $\bar{U}^+$) is a minimum (resp. a maximum) on a finite non empty set, it is well-defined.\\
It is direct to see that $\bar{U}^-$ and $\bar{U}^+$ satisfy \eqref{eq:cond:Umono}, \eqref{eq:condUA1} and \eqref{eq:condUA2}. Condition (\ref{eq:cond:Uboundary}.$a$) is straightforward since $f_{\ell,m}^-$ is positive for all $(\ell,m) \in \Omega^-$.\\
Finally, for any $\ell \in \mathbb{N}$, the sequence $(f_{\ell,m}^+)_{m > \ell}$ is decreasing, so we have for all $(\ell,m) \in \Omega^+$,
$$ \bar{U}^+_{\ell,m} = \max_{ \ell' \in \{ 0, \ldots, \ell \}} f_{\ell',m}^+ = \max_{\ell' \leq \ell} \max_{i \in S_{\ell'}} Q_{i,m:\infty}.$$
Then since the sequence $(Q_{i,m:\infty})_m$ goes to $0$ and the two above maximum are taken on finite sets, the condition (\ref{eq:cond:Uboundary}.$b$) is fulfilled.\\
By \Cref{lem:2.8}, $\Phi^{-1}(\bar{U}^-,\bar{U}^+)$ is a transition-rate matrix satisfying Assumptions \ref{assum:second_index} and \ref{assum:first_index}.\\
It remains to prove the optimality of the matrix $\Phi^{-1}(\bar{U}^-,\bar{U}^+)$. Let $\widetilde{Q}$ be a transition-rate matrix satisfying Assumptions \ref{assum:second_index} and \ref{assum:first_index} and suppose that there exists $(\ell,m) \in \mathbb{N}^2$ such that
\begin{equation}
\label{eq:contradi1}
\widetilde{Q}_{\ell,0:m} > \Phi^{-1}(\bar{U}^-,\bar{U}^+)_{\ell,0:m}.
\end{equation}
Two cases must be considered.
\begin{itemize}
\item If $m < \ell$, with \Cref{lem:bijec}, $\Phi^{-1}(\bar{U}^-,\bar{U}^+)_{\ell,0:m} = \bar{U}^-_{\ell,m} = \min\limits_{\ell' \in \{ m+1, \ldots, \ell \} } f_{\ell',m }^-$.\\
Thus there exists an integer $\ell_0 \in [m+1,\ell]$ such that $\Phi^{-1}(\bar{U}^-,\bar{U}^+)_{\ell,0:m} = f_{\ell_0,m }^-$. Moreover, with Assumptions \ref{assum:second_index} and \ref{assum:first_index}, and since $\ell_0 > m$ we have
\begin{equation} \label{eq:contradi2}
\widetilde{Q}_{\ell,0:m} \leq \widetilde{Q}_{\ell_0,0:m} \leq f_{\ell_0,m }^-=\Phi^{-1}(\bar{U}^-,\bar{U}^+)_{\ell,0:m}.
\end{equation}
There is a contradiction between \eqref{eq:contradi1} and \eqref{eq:contradi2}.
\item If $m \geq \ell$, with \Cref{lem:bijec}, $\Phi^{-1}(\bar{U}^-,\bar{U}^+)_{\ell,0:m} = - \bar{U}^+_{\ell,m+1} = -\max\limits_{\ell' \in \{ 0, \ldots, \ell \} } f_{\ell',m+1}^+$.\\
Thus there exists an integer $\ell_0 \in [0,\ell]$ such that $\Phi^{-1}(\bar{U}^-,\bar{U}^+)_{\ell,0:m} = -f_{\ell_0,m+1 }^+$. Moreover, with Assumptions \ref{assum:second_index} and \ref{assum:first_index}, and since $\ell_0 \leq \ell \leq m$ we have
\begin{equation}
\label{eq:contradi3}
\widetilde{Q}_{\ell,0:m} \leq \widetilde{Q}_{\ell_0,0:m} = - \widetilde{Q}_{\ell_0,m+1:\infty} \leq -f_{\ell_0,m+1 }^+.
\end{equation}
Again, there is contradiction between \eqref{eq:contradi1} and \eqref{eq:contradi3}.
\end{itemize}
So we have proven by contradiction that
\begin{equation*}
\widetilde{Q}_{\ell,0:m}
\leq \Phi^{-1}(\bar{U}^-,\bar{U}^+)_{\ell,0:m}, \quad
\forall (\ell,m) \in \mathbb{N}^2.
\end{equation*}
This concludes the proof.
\end{proof}

\subsection{Uniform bounds of solutions of the master equation}

In this section, we derive from our main result (\Cref{thm:main_upper-bound}) a comparison on the probability distribution of processes in a family parameterized by a parameter $\theta$.\\

Let $\left(X^{\theta}\right)_{\theta \in \Theta}$ be a family of processes parameterized by a set $\Theta$ of suitable parameters. For any $\theta \in \Theta$, the master equation determines the probability distribution of $X^{\theta}$,
$$ \forall t \geq 0,~ \frac{\d p^{\theta}}{\d t}(t) = p^{\theta}(t) Q^{\theta}, $$
where $p^{\theta}(t) = (p^{\theta}_i(t))_{i \in \mathcal{S}}$ is a row vector such that $p^{\theta}_i(t) = \mathbb{P}\left(X^{\theta}_t = i | X^{\theta}_0\right)$.

In the following, the solution of the master equation linked to the upper-bounding process $Y$ (i.e. for the matrix $\widetilde{Q}$) will be denoted by $\widetilde{p}_t$. Then we can state an inequality for the probability distributions of the family $\left(X^{\theta}\right)_{\theta \in \Theta}$ and of $Y$.

\begin{thm}\label{thm:sol_cme_upper}
Assume that for any $\theta \in \Theta$, the pair $\left(Q^{\theta}, \widetilde{Q}\right)$ satisfies the Assumptions \ref{assum:second_index} and \ref{assum:first_index}.\\
Assume further that
$$ \forall \ell \in \mathbb{N},~ \widetilde{p}_{0:\ell}(0) \leq \inf_{\theta \in \Theta} p^{\theta}_{0:\ell}(0). $$
Then for any $t \geq 0$ and $\ell \in \mathbb{N}$,
\begin{equation}\label{eq:ineq_cme_upper}
\tilde{p}_{0:\ell}(t) \leq \inf_{\theta \in \Theta} p^{\theta}_{0:\ell}(t).
\end{equation}
The above inequality means that the process $\tilde{p}$ keeps less mass in the $\ell$ first classes than $p^{\theta}$ that is the mass go faster to infinity for the process $\tilde{p}$.
\end{thm}

\textbf{Remarks:}\begin{itemize}
\item If the set $\Theta$ is a singleton, \Cref{thm:sol_cme_upper} is a weaker result than \Cref{thm:main_upper-bound} since it provides an inequality between the probability distributions $p$ and $\tilde{p}$.\\
However, \Cref{thm:main_upper-bound} cannot be generalised to a family of processes $\left(X^{\theta}\right)_{\theta \in \Theta}$.
\item Since $\widetilde{p}$ and $p^{\theta}$ are probability distributions, the inequality (\ref{eq:ineq_cme_upper}) can be rewritten in terms of the distribution tails:
$$ \forall t \geq 0,~ \forall \ell \in \mathbb{N},~ \inf_{\theta \in \Theta} p^{\theta}_{\ell:\infty}(t) \leq \widetilde{p}_{\ell:\infty}(t). $$
\end{itemize}

\begin{proof}
Let $\theta \in \Theta$, with \Cref{thm:main_upper-bound}, for any $t \geq 0$,
$$ \cl{X^{\theta}_t} \leq Y_t \text{ a.s}. $$
So for any $\ell \in \mathbb{N}$, we have
$$ \left\{Y_t \leq \ell \right\} \subseteq \left\{ \cl{X^{\theta}_t } \leq \ell \right\} .$$
By definition of the probability distributions,
$$ \forall \theta \in \Theta,~ \forall t \geq 0,~ \forall \ell \in \mathbb{N},~ \widetilde{p}_{0:\ell}(t) \leq p^{\theta}_{0:\ell}(t) .$$
By taking the infimum on the set $\Theta$, we obtain the result.
\end{proof}

\section{Construction of a lower-bounding process}\label{sec:lower}

In this section, by adapting proofs and formulas, we give simple conditions on the transition-rate matrix $\widetilde{Q}$ of a lower-bounding process and then construct such a matrix.

To be able to construct the lower-bounding process, the following hypothesis on the transition-rate matrix $Q$ is necessary:

\begin{equation}\label{eq:hyp_supp_lower}
\forall m \in \mathbb{N},~\exists B_m \geq 0,~ \forall i \in \mathcal{S},~ \cl{i} > m \Rightarrow Q_{i,m} \leq B_m.
\end{equation}
This hypothesis says that the rate of all reactions leading to a decrease of classes and ending in a class $m$ are uniformly bounded with respect to the starting class. It is directly verified when the number of reactions is finite.

\subsection{Main result}

\begin{thm} \label{thm:main_lower-bound}
Let $Q \in \trmat{\mathcal{S}}$ satisfying \Cref{eq:hyp_supp_lower} and let $\widetilde{Q} \in \trmat{\mathbb{N}}$. Let $p^0$ be a distribution on $S$. Assume the two following properties:
\begin{itemize}
\item[] \begin{enumerate}[label = \textbf{\upshape (B\arabic*)}, ref = \textbf{\upshape B\arabic*}
]
\item \label{assum:second_index_lower}
$\forall \ell$, $m \in \mathbb{N}$,~ $\forall i \in S_{\ell}$,~ $\widetilde{Q}_{\ell,0:m} \geq Q_{i,0:m}$,
\item \label{assum:first_index_lower}
$\forall \ell$, $m \in \mathbb{N}$ such that $\ell \neq m$,~ $\widetilde{Q}_{\ell+1,0:m} \leq \widetilde{Q}_{\ell,0:m}$.
\end{enumerate}
\end{itemize}
Then there exist two continuous-time Markov chains $X=(X_t)_{t \geq 0}$ and $Y= (Y_t)_{t \geq 0}$, defined on $\mathcal{S}$ and $\mathbb{N}$ respectively, with transition-rate matrices $Q$ and $\widetilde{Q}$, such that $X_0$ is distributed according to $p^0$, and such that $Y$ is a lower-bounding process of $X$.
\end{thm}

As in the previous section, we can give an interpretation of the two assumptions. Firstly note that the Assumption \ref{assum:first_index_lower} is the same as Assumption \ref{assum:first_index}. Conversely, the inequality is inverted between Assumptions \ref{assum:second_index} and \ref{assum:second_index_lower}.

\begin{rem} \label{rem:assumptionB1}
Assumption \ref{assum:second_index_lower} essentially says that $Y$ decreases at a faster rate than $X$ moves from states to other states that belong to a lower class. Vice-versa, $Y$ increases at a smaller rate than $X$ moves from states to other states that belong to a higher class. This can be formalized as follows.
\begin{itemize}
\item When $m < \ell$, the assumption is equivalent to:
\begin{equation*}
\widetilde{Q}_{\ell,0:m}
\geq \min_{i \in S_\ell} \, Q_{i,0:m}.
\end{equation*}
The quantity $\widetilde{Q}_{\ell,0:m}$ is the rate at which $Y$ moves from $\ell$ to classes that are smaller or equal to $m$ and the quantity $\min_{i \in S_\ell} \, Q_{i,0:m}$ is the minimal rate at which $X$ transitions from a state in class $\ell$ to another state in a class lower or equal to $m$.
\item When $m \geq \ell$, the assumption is equivalent to:
\begin{equation*}
\widetilde{Q}_{\ell,m+1:\infty}
\leq \max_{i \in S_\ell}\, Q_{i,m+1:\infty}.
\end{equation*}
The quantity $\widetilde{Q}_{\ell,m+1:\infty}$ is the rate at which $Y$ moves from $\ell$ to classes that are are strictly greater than $m$ and the quantity $\max_{i \in S_\ell}\, Q_{i,m+1:\infty}$ is the maximal rate at which $X$ transitions from a state in class $\ell$ to another state in a class strictly higher than $m$.
\end{itemize}
\end{rem}

The interpretation of the Assumption \ref{assum:first_index_lower} has already been stated in Remark \ref{rem:assumptionA2} since this assumption is the same as Assumption \ref{assum:first_index}.

\subsection{Construction of a lower-bounding transition rate-matrix}\label{subsec:construction_Qtilde_lower}

By adapting the approach and formulas of \Cref{subsec:construction_Qtilde}, we find a transition-rate matrix $\widetilde{Q}$ that satisfies Assumptions \ref{assum:second_index_lower} and \ref{assum:first_index_lower}.\\

With the same notations as in \Cref{subsec:construction_Qtilde}, we define the two doubly-indexed sequences $\widehat{f}^- = \big(\widehat{f}_{\ell,m}^- \big)_{(\ell,m) \in \Omega^-}$ and $\widehat{f}^+ = \big(\widehat{f}_{\ell,m}^+ \big)_{(\ell,m) \in \Omega^+}$ as
\begin{align}
\widehat{f}_{\ell,m}^-= {} \max_{i \in S_{\ell}} \, Q_{i,0:m} \quad \text{and} \quad
\widehat{f}_{\ell,m}^+ = {} \min_{i \in S_{\ell}} \, Q_{i,m:\infty}. \label{eq:def_f_lower}
\end{align}

We define next two candidates $\hat{U}^-$ and $\hat{U}^+$ for the image of the matrix $\widetilde{Q}$ by the function $\Phi$ (see \Cref{lem:bijec}) as follows:
\begin{equation}\label{eq:optimal_lower}
\begin{array}{rll}
\hat{U}^-_{\ell,m}
= & \sup\limits_{\ell' \in \{ \ell, \ell+1, \ldots\}} \widehat{f}_{\ell',m}^- & \quad \forall (\ell,m) \in \Omega^-, \\
\hat{U}^+_{\ell,m} = & \min\limits_{\ell' \in \{ \ell, \ldots, m-1 \} } \widehat{f}_{\ell',m}^+ & \quad \forall (\ell,m) \in \Omega^+.
\end{array}
\end{equation}

\begin{prop}
The quantities $\hat{U}^-$ and $\hat{U}^+$ are well-defined and are in the set $\mathbf{U}$.\\
Then $\Phi^{-1}(\hat{U}^-,\hat{U}^+)$ is a transition-rate matrix satisfying Assumptions \ref{assum:second_index_lower} and \ref{assum:first_index_lower}.\\
Moreover, this transition-rate matrix $\Phi^{-1}(\hat{U}^-,\hat{U}^+)$ is optimal in the following sense: any other transition-rate matrix $\widetilde{Q}$ satisfying \Cref{eq:hyp_supp_lower} and Assumptions \ref{assum:second_index_lower} and \ref{assum:first_index_lower} is necessarily such that
\begin{equation*}
\Phi^{-1}(\hat{U}^-, \hat{U}^+)_{\ell,0:m} \leq \widetilde{Q}_{\ell,0:m}, \quad
\forall (\ell,m) \in \mathbb{N}^2.
\end{equation*}
\end{prop}

\begin{proof}
Since $\hat{U}^+$ is a minimum on a finite set, it is well-defined. With \Cref{eq:hyp_supp_lower},
$$ \forall (\ell,m) \in \Omega^-,~ \widehat{f}_{\ell,m}^- \leq B_m, $$
where $B_m$ is independent of $\ell$. Thus $\hat{U}^-$ is well-defined.\\
The rest of the proof is an adaptation of the proof of \Cref{prop:optimal_upper}.
\end{proof}

\section{Applications}\label{sec:appli}

In this section, we will detail how different properties of the original process can be derived from the lower and upper-bounding processes. As an example, we will use the following chemical reaction network composed of three species $X_1$, $X_2$ and $X_3$

$$ \left.
\begin{array}{rclrclrclc}
\multirow{2}{*}{$\emptyset$} & \multirow{2}{*}{$\stackrel{b_1(x_2+x_3)+b_2}{\longrightarrow}$} & \multirow{2}{*}{$X_1$ \qquad }& 2X_1 & \stackrel{\alpha x_1 (x_1-1)}{\longrightarrow} & 3X_2 \phantom{qqqq}  & \multirow{2}{*}{$X_i$} & \multirow{2}{*}{$\stackrel{d_i x_i}{\longrightarrow}$} & \multirow{2}{*}{$\emptyset$} & \multirow{2}{*}{~ for $1\leq i\leq 3$} \\
& & & X_1 + X_2 & \stackrel{\beta x_1 x_2}{\longrightarrow} & X_3 & & & & \\
\end{array}
\right. $$
with all parameters ($b_1$, $b_2$, $\alpha$, $\beta$, $(d_i)_{1\leq i\leq 3}$) strictly positive. In the following, we will call $\theta$ the parameter vector.\\

This network is composed of six reactions: one reaction creating the species $X_1$, two second-order reactions for the creation of species $X_2$ and $X_3$, and three reactions for the degradation of each species. The dynamics of each species depend on those of the other two due to the reaction rates chosen. In particular, the creation of the species $X_1$ is positively regulated by the sum of the number of molecules $X_2$ and $X_3$.

The motivation for choosing this reaction network as an example is that it poses significant challenges: the non-linearity of the reaction propensities and the dimensionality of the network make it difficult to analyze the dynamics and to establish properties of the CTMC. Notably, the creation of  molecules of $X_2$ at a quadratic rate implies that molecule numbers could potentially diverge to infinity in  finite time for some choices of model parameters.\\

Let $X$ be the continuous-time Markov chain on $\mathbb{N}^3$ associated to this network. In the following, we will apply our method with different sets of classes in order to derive results on the original process. For the upper-bounding processes, we will consider the two following sets of classes

\begin{equation}\label{eq:def_upper_classes}
\begin{array}{cccc}
\forall \ell \in \mathbb{N},~ &S^{+}_{\ell} = \{ 2x_1 + x_2 + x_3 = \ell \} &\text{ and }&  \tilde{S}^{+}_{\ell} = \{ x_1 + x_2 + x_3 = \ell \}.
\end{array}
\end{equation}

The partition $(\tilde{S}^{+}_{\ell})_{\ell \in \mathbb{N}}$ is the naive choice of classes that groups states according to the total number of molecules. Contrary to that, the partition $(S^{+}_{\ell})_{\ell \in \mathbb{N}}$ has a linear form too but gives more weight to the molecules of species $X_1$. The main motivation for weighting more the species $X_1$ is to take into account the structure of the reaction network by ensuring that the reaction $2X_1 \rightarrow 3X_2$ decreases classes.\\
For the lower-bounding process, we will consider only one set of classes
\begin{equation}\label{eq:def_lower_classes}
\begin{array}{cc}
\forall \ell \in \mathbb{N},~ &S^{-}_{\ell} = \{ 2x_1 + 2x_2 + 5x_3 = \ell \}.
\end{array}
\end{equation}
In the same spirit as for the two previous partitions, the weights in the partition $(S^{-}_{\ell})_{\ell \in \mathbb{N}}$ are chosen such that the two reactions $2X_1 \rightarrow 3X_2$ and $X_1+X_2 \rightarrow X_3$ increase classes.\\

In a first part, we will illustrate the usefulness of our method to obtain conditions on the parameters for non-explosion and ergodicity of the multi-dimensional CTMC $X$. Then, in a second part, we will compare the optimal upper-bounding processes obtained with the two partitions $(S^{+}_{\ell})_{\ell \in \mathbb{N}}$ and $(\tilde{S}^{+}_{\ell})_{\ell \in \mathbb{N}}$ introduced before which then allows us to discuss the importance of the choice of the classes for the coupling and for the study of the stationary distribution. Finally, we will bound the error made in numerically solving the master equation on a truncated state space.

\subsection{Asymptotic behavior}\label{subsec:asymptotic_behavior}

\subsubsection{Theoretical results}

Let us recall some links between the asymptotic behavior of a CTMC $X$ and hitting times. Let $X$ be an $\textbf{irreducible}$ continuous-time Markov chain on a state space $\mathcal{S}$.\\

The sequence $J = (J_n)_{n \in \mathbb{N}}$ of jump times of $X$ are defined by
$$J_0 = 0 \text{~ and for any $n \geq 1$, } J_n = \inf \{t \geq J_{n-1},~ X_t \ne X_{J_{n-1}} \},$$
where $\inf \emptyset = \infty$ by convention.\\
The process $X$ is explosive if and only if $\mathbb{P}(\sup_n J_n < \infty) > 0$.\\

For any state $i \in \mathcal{S}$, the first hitting time of $i$ is the first time at which the chain reaches state $i$ (excluding time $0$), that is
$$ T^{X}_i = \inf\left(t>0, X_t = i \right), $$
with the convention that $T^{X}_i = \infty$ if the chain never reaches state $i$.\\

The asymptotic behavior of the irreducible continuous-time Markov chain $X$ can be described with the hitting times:
\begin{enumerate}
\item $X$ is transient $\Longleftrightarrow$ there exists $i \in \mathcal{S}$ such that $\mathbb{P}\left(T^{X}_i=\infty\right) >0 $.
\item $X$ is recurrent $\Longleftrightarrow$ there exists $i \in \mathcal{S}$ such that $\mathbb{P}\left(T^{X}_i=\infty\right)=0 $.\\
We distinguish between two types of recurrent processes:
\begin{itemize}
\item $X$ is null-recurrent $\Longleftrightarrow$ there exists $i \in \mathcal{S}$ such that $\mathbb{E}\left(T^{X}_i \right) =\infty $.
\item $X$ is positive recurrent $\Longleftrightarrow$ there exists $i \in \mathcal{S}$ such that $\mathbb{E}\left(T^{X}_i \right) < \infty $.
\end{itemize}
\end{enumerate}
Since $X$ is irreducible, the existence of a stationary distribution is equivalent to positive recurrence. Note that only transient process can be explosive, so a non-explosive process can be transient, null-recurrent or positive recurrent.\\

In the following, the coupling presented in the previous sections is used in order to determine the asymptotic behavior of $X$. Let $Z$ be a lower-bounding process for classes $\left(S^{-}_\ell\right)_{\ell \in \mathbb{N}}$ and $Y$ a upper-bounding process for classes $\left(S^{+}_\ell\right)_{\ell \in \mathbb{N}}$ (see \Cref{sec:upper,sec:lower}), that is for any $t \geq 0$,
$$
\left\{
\begin{array}{cccc}
Z_t &\leq& \text{cl}^{+}(X_t), & \text{for $\left(S^{-}_{\ell}\right)_{\ell \in \mathbb{N}}$} \\
\text{cl}^{-}(X_t) &\leq& Y_t, &  \text{for $\left(S^{+}_{\ell}\right)_{\ell \in \mathbb{N}}$} \\
\end{array}
\right. ,
$$
where the functions $\text{cl}^{+}$ (resp. $\text{cl}^{-}$) is the function defined in \Cref{def:fct_cl} for classes $\left(S^{+}_\ell\right)_{\ell \in \mathbb{N}}$ (resp. for $\left(S^{-}_\ell\right)_{\ell \in \mathbb{N}}$).

\begin{thm}\label{thm:statio_distri}
Suppose that processes $Y$ and $Z$ are irreducible. Then according to the asymptotic behavior of the lower and upper-bounding processes $Z$ and $Y$, we can deduce properties on the asymptotic behavior of the process $X$ presented in the \Cref{tab:asympto_behavior}:

\begin{table}[H]
\centering
\settowidth\rotheadsize{Recurrent/}
\begin{NiceTabular}{wc{0.8cm}wc{2.9cm}wc{2.4cm}ccc}[hvlines]
\CodeBefore
\cellcolor[HTML]{ADD8E6}{3-3,4-4,5-4,6-5,7-5,8-6}
\cellcolor[HTML]{000000}{3-4,6-6,7-6}
\rectanglecolor[HTML]{000000}{3-5}{5-6}
\Body
\hline
\Block{2-2}{\diagbox{\quad \rule[-3mm]{0pt}{10mm}Lower process $Z_t$}{\rule[-10mm]{0pt}{15mm}Upper process $Y_t$ \quad}} && \Block{1-2}{Transient} && \Block{1-2}{Recurrent} \\
\cline{3-6}
& & Explosive & Non-explosive & Null-recurrent & Positive recurrent \\
\hline
\Block{3-1}{\rothead{\text{Transient}}} & Explosive & Explosive & & & \\
\cline{2-6}
& \Block{2-1}{Non-explosive} & \Block{2-1}{Transient \\ (Explosive or not)} & \Block{2-1}{Transient and \\ non-explosive} & & \\
& & & & \\
\hline
\Block{3-1}{\rothead{Recurrent}} & \Block{2-1}{Null-recurrent} & \Block{2-1}{Transient or \\ null-recurrent} & \Block{2-1}{Non-explosive} & \Block{2-1}{Null-recurrent} & \\
& & & & \\
\cline{2-6}
& Positive recurrent & No information & Non-explosive & Recurrent & Positive recurrent\\
\hline
\end{NiceTabular}
\caption{Asymptotic behavior of $X$ according to that of the lower and upper-bounding processes $Z$ and $Y$. Black regions represent impossible combinations. In the blue regions, the asymptotic behavior of $X$ can be fully deduced from the bounding processes. In the white regions, we only have access to partial information.}\label{tab:asympto_behavior}
\end{table}
\end{thm}

\begin{proof}
The results on the asymptotic behavior of the process $X$ presented in the \Cref{tab:asympto_behavior} are equivalent to the following set of implications:
\begin{equation*}
\begin{cases}
\begin{array}{rcl}
Y \text{ non-explosive} & \Longrightarrow & X \text{ non-explosive,}\\
Y \text{ recurrent} & \Longrightarrow & X \text{ recurrent,}\\
Y \text{ positive recurrent} & \Longrightarrow & X \text{ positive recurrent.}
\end{array}
\end{cases}
\quad \quad
\begin{cases}
\begin{array}{rcl}
Z \text{ explosive} & \Longrightarrow & X \text{ explosive,}\\
Z \text{ transient} & \Longrightarrow & X \text{ transient,}\\
Z \text{ null-recurrent} & \Longrightarrow & X \text{ non positive recurrent.}
\end{array}
\end{cases}
\end{equation*}
Since $Y$ is an upper-bounding process of $X$, for any $i \in S^{+}_0$,
$\{T^{X}_i = \infty \} \Longrightarrow \{T^{Y}_0 = \infty\} $, so
\begin{equation}\label{eq:ineq_X_Y}
\mathbb{P}\left(T^{X}_i = \infty \right) \leq \mathbb{P}\left(T^{Y}_0 = \infty \right).
\end{equation}
The same inequality holds for the expectation,
\begin{align}
\mathbb{E}\left( T^{X}_i\right) &= \int_0^{\infty} \mathbb{P} \left(T^{X}_i \geq t \right) dt \nonumber \\
&\leq \int_0^{\infty} \mathbb{P} \left(T^{Y}_0 \geq t \right) dt \nonumber \\
&\leq \mathbb{E}\left( T^{Y}_0\right). \label{eq:ineq_X_Y_2}
\end{align}
With the same calculations, we obtain corresponding inequalities for the lower-bounding process $Z$:
\begin{equation}\label{eq:ineq_X_Z}
\forall i' \in S^{-}_0,~~ \begin{cases} \begin{array}{ccl}
\mathbb{P}\left(T^{Z}_0 = \infty \right) &\leq & \mathbb{P}\left(T^{X}_{i'} = \infty \right) \\
\mathbb{E}\left( T^{Z}_0\right) &\leq & \mathbb{E}\left( T^{X}_{i'}\right)
\end{array} \end{cases}.
\end{equation}
Then the implications on the left follow from \Cref{eq:ineq_X_Y,eq:ineq_X_Y_2}, since we suppose that processes $X$ and $Y$ are irreducible. The implications on the right follow from \Cref{eq:ineq_X_Z} since we suppose that processes $X$ and $Z$ are irreducible.
\end{proof}


\subsubsection{Application of the method to the toy model}\label{sssec:numeric_stationary}

In this subsection, we compute respectively the transition-rate matrices $Q^{-}$ and $Q^{+}$ of the optimal lower and upper-bounding processes $Z$ and $Y$ for two different sets of classes and then deduce results on the asymptotic behavior of the process $X$.\\

Let us start with the lower-bounding process $Z$ and recall the set of classes defined by \Cref{eq:def_lower_classes},
$$ \forall \ell \in \mathbb{N}, ~~ S^{-}_{\ell} = \{x \in \mathbb{N}^3,~ 2x_1 + 2x_2 + 5x_3 = \ell \}. $$

\begin{prop}\label{prop:matrix_lower}
Using the results of \Cref{sec:lower}, the transition-rate matrix $Q^{-}$ of the process $Z$ can be expressed as
$$ Q^{-} = \Phi(\hat{U}^-, \hat{U}^+), $$
where $(\hat{U}^-_{\ell, m})_{(\ell,m) \in \Omega^-}$ and $(\hat{U}^+_{\ell, m})_{(\ell,m) \in \Omega^+}$ are defined by \Cref{eq:optimal_lower}.\\
Since we will focus on the long-term behavior and to keep formulas simple, the only values that matter are $(\hat{U}^-_{\ell, m})_{(\ell,m) \in \Omega^-}$ and $(\hat{U}^+_{\ell, m})_{(\ell,m) \in \Omega^+}$ for $\ell$ large enough.\\
Thus, we have for any $\ell$ large enough,
\begin{equation*}
\hat{U}^-_{\ell, m}
= \begin{cases}
\begin{array}{cl}
0 & \text{ if $m < \ell -5$}, \\
\frac{d_3}{5}(m+5)  & \text{ if $\ell - 5 \leq m \leq l-3$}, \\
B(\theta)(m+2) & \text{ if $m \in \{\ell -2, \ell -1 \}$}.
\end{array}
\end{cases}
\end{equation*}
\begin{equation*}
\hat{U}^+_{\ell, m}
= \begin{cases}
\begin{array}{cl}
\frac{b_1}{5} \ell + A(\theta) & \text{ if $\ell + 1 \leq m \leq \ell+2$}, \\
0  & \text{ if $\ell +3 \leq m $}.
\end{array}
\end{cases}
\end{equation*}
with $A(\theta) = b_2 - \frac{(\alpha + \frac{2}{5}b_1)^2}{4\alpha}$ and $B(\theta) = \max\Big(\frac{d_1}{2}, \frac{d_2}{2}, \frac{d_3}{5} \Big)$.
\end{prop}

\begin{proof}
The proof of \Cref{prop:matrix_lower} is detailed in \Cref{ssec:app_1}.
\end{proof}

Note that the previous proposition gives the expression of the lower-bounding process only for large integers $\ell$ corresponding to large classes of the initial process. The computation of the exact expression of the transition rates is possible for any integer $\ell$ but leads to many case disjunctions depending on the parameter values. Since transition rates for small classes are not relevant for the study of the asymptotic behavior of the process, we refrain from listing them here.\\

Now, we construct the transition-rate matrix of the upper-bounding process $Y$. Let us recall the following set of classes defined by \Cref{eq:def_upper_classes},
$$ \forall \ell \in \mathbb{N}, ~~ S^{+}_{\ell} = \{ 2x_1 + x_2 + x_3 = \ell \}.$$

\begin{prop}\label{prop:upper_processes_def1}
Using the results of \Cref{sec:upper}, the transition-rate matrix $Q^{+}$ of the process $Y$ verifies, for any $\ell$ large enough,
\begin{center}
\begin{tabular}{cr}
$Q^{+}_{\ell,\ell-1} = \min(d_2,d_3)\ell - C(\theta)$, & $Q^{+}_{\ell,\ell} = -Q^{+}_{\ell,\ell-1}-Q^{+}_{\ell,\ell+2}$, \\
$Q^{+}_{\ell,\ell+2} = b_1 \ell +b_2$, & $Q^{+}_{\ell,m} = 0$ \text{ if $m \notin \{\ell-1, \ell+2\}$}, \\
\end{tabular}
\end{center}
where $C(\theta) = \frac{(\alpha+2\min(d_2,d_3)-d_1)^2}{4\alpha}\mathds{1}_{\alpha+2\min(d_2,d_3)-d_1 \geq 0}(\theta)$.\\
\end{prop}

\begin{proof}
The proof of \Cref{prop:upper_processes_def1} is detailed in \Cref{ssec:app_2}.
\end{proof}

In order to be able to apply \Cref{thm:statio_distri} to obtain results on the asymptotic behavior of the process $X$, it only remains to study the asymptotic behavior of the two one-dimensional CTMCs $Y$ and $Z$. Based on \cite{xu_full_2020}, we defined the following quantities:
\begin{itemize}
\item for the process $Y$,
$$ B'_1 = 2b_1 -  \min(d_2,d_3) \text{~~~ and ~~~} B'_3 = 2b_2+C(\theta)- \frac{3}{2}\min(d_2,d_3), $$
\item and for the process $Z$,
$$ B_1 := \lim_{\ell \to \infty} \frac{ \hat{U}^+_{\ell,\ell+1:\infty} - \hat{U}^-_{\ell,0:\ell-1}}{\ell} = \frac{2}{5}b_1 - \frac{3}{5}d_3 - 2B(\theta),$$
$$ B_2 := \lim_{\ell \to \infty} \left( \hat{U}^+_{\ell,\ell+1:\infty} - \hat{U}^-_{\ell,0:\ell-1} - B_1 \ell \right) = 2A(\theta) - \frac{3}{5}d_3 - B(\theta), $$
$$B_3 := B_2 - \frac{1}{2} \lim_{\ell \to \infty} \frac{\sum\limits_{k=1}^{\infty} (2k-1) \hat{U}^+_{\ell,\ell+k} + \sum\limits_{k=1}^{\ell} (2k-1) \hat{U}^-_{\ell,\ell-k}}{\ell} = 2A(\theta) - \frac{2}{5}b_1 -\frac{27}{10}d_3 - 3B(\theta). $$
\end{itemize}

Parameters $B_1$ and $B'_1$ represent the average jump size of the processes. $B_2$ is the average jump size of the process if the linear contributions are removed. The expression of the average jump size of the process $Y$ ($B'_2$) is not detailed here since this parameter is always strictly positive, meaning that it has no impact on the asymptotic behavior of the process. The sign of these five parameters determine the asymptotic behavior of the processes $Y$ and $Z$.

\begin{table}[H]
\centering
\settowidth\rotheadsize{Recurrent/}
\begin{NiceTabular}{|c|c|c|c|c|}
\CodeBefore
\cellcolor[HTML]{000000}{4-5,5-5}
\cellcolor[HTML]{ADD8E6}{3-3,4-4,5-4,6-5}
\rectanglecolor[HTML]{000000}{3-4}{3-5}
\Body
\hline
\Block{2-2}{\diagbox{\quad \rule[-4mm]{0pt}{10mm}Lower process $Z_t$}{\rule[-10mm]{0pt}{15mm}Upper process $Y_t$ \quad}} && Transient and non-explosive & Null-recurrent & Recurrent \\
\cline{3-5}
& & $B'_1 >0$ or $\begin{cases} \begin{array}{rl}
B'_1 &= 0  \\
B'_3 &> 0
\end{array} \end{cases}$ & $\begin{cases} \begin{array}{rl}
B'_1 &= 0  \\
B'_3 &\leq 0
\end{array} \end{cases}$ & $B'_1 <0$ \\
\hline
\rothead{\text{Transient}} & $B_1 >0$ or $\begin{cases} \begin{array}{c}
B_1 = 0  \\
B_3 > 0
\end{array} \end{cases}$ & Transient and non-explosive & & \\
\hline
\Block{3-1}{\rothead{Recurrent}} & \Block{2-1}{ $\begin{cases} \begin{array}{c}
B_1 = 0  \\
B_2 \geq 0 \\
B_3 \leq 0
\end{array} \end{cases}$} & \Block{2-1}{Non-explosive} & \Block{2-1}{Null-recurrent} & \\
& & & & \\
\cline{2-5}
& $\begin{cases} \begin{array}{c}
B_1 = 0  \\
B_2 < 0
\end{array} \end{cases}$ or $B_1 <0$ & Non-explosive & Recurrent & Positive recurrent\\
\hline
\end{NiceTabular}
\caption{Asymptotic behavior of $X$ according to the parameter values. The black regions represent impossible parameter combinations. In the blue regions, the asymptotic behavior of $X$ can be fully deduced from the bounding processes. In the white regions, we only have access to partial information.}\label{tab:asympto_behavior_parameters}
\end{table}

\Cref{tab:asympto_behavior_parameters} illustrates the kind of condition that we can obtain with the construction of the lower and upper-bounding processes. In this case, the processes $Y$ and $Z$ are non-explosive for all parameter values so the corresponding column and row in \Cref{tab:asympto_behavior} have been removed.\\

\subsection{Importance of the choice of classes}

Let us recall the following partition defined by \Cref{eq:def_upper_classes},

$$ \forall \ell \in \mathbb{N},~ \tilde{S}^{+}_{\ell} = \{ x_1 + x_2 + x_3 = \ell \}. $$

In this partition $\tilde{S}^{+}$, the states are ordered by the total number of molecules. This partition is the simplest partition but does not take into account the structure of the reaction network.\\

Let $\tilde{Q}^+$ be the matrix of the optimal upper-bounding process $\tilde{Y}$ related to the partition $(\tilde{S}^{+}_{\ell})_{\ell\in \mathbb{N}}$. Then, we have

\begin{prop}\label{prop:upper_processes_def2}
The matrix $\tilde{Q}^+$ of the optimal upper-bounding process $\tilde{Y}$ for the set of classes $(\tilde{S}^{+}_{\ell})_{\ell\in \mathbb{N}}$ verifies, for any $\ell \geq 0$,
$$ \tilde{Q}^{+}_{\ell,\ell-1} = d_2\ell ~, ~~~~ \tilde{Q}^{+}_{\ell,\ell+1} = \alpha \ell^2 - \alpha \ell +b_2, $$
$$ \tilde{Q}^{+}_{\ell,\ell} = -\tilde{Q}^{+}_{\ell,\ell-1}-\tilde{Q}^{+}_{\ell,\ell+1}. $$
For any set of parameters $\theta$, the process $\tilde{Y}$ is explosive.
\end{prop}

\begin{proof}
The calculations of the matrix $\tilde{Q}^+$ is detailed in \Cref{ssec:app_3}.\\
With the expressions of the birth and death rates $\tilde{Q}^{+}_{l,l+1}$ and $\tilde{Q}^{+}_{l,l-1}$ and Theorems (3.1), (3.3) and (3.7) of \cite{xu_full_2020}, we obtain directly thean the process $\tilde{Y}$ is explosive.
\end{proof}

One can see that using the classes $\tilde{S}^+$ leads to a different optimal upper-bounding process compared to the process in \Cref{prop:upper_processes_def1} that was obtained with the classes $S^+$. Concretely, we obtain a process with a quadratic birth rate and linear death rate as opposed to the linear birth and death rates in \Cref{prop:upper_processes_def1}.\\

According to \Cref{tab:asympto_behavior}, when the upper-bounding process is explosive (first column of \Cref{tab:asympto_behavior}), the original process $X$ can have any asymptotic behavior: explosive, transient, null-recurrent or positive recurrent. Thus the naive choice of classes gives no information about the asymptotic behavior of the process $X$. On the contrary, the upper-bounding process $Y$ used in \Cref{sssec:numeric_stationary} gives relevant conditions on the asymptotic behavior of the process $X$ (see \Cref{tab:asympto_behavior_parameters}). Thus, the choice of the set of classes influences the quality of the results obtained on the stationary distribution.\\

\subsubsection{Coupling with two sets of classes}

In this subsection, numerical simulations are used to illustrate the coupling used in the proof of \Cref{thm:main_upper-bound} and the difference between the two upper-bounding processes $Y$ and $\tilde{Y}$ (see \Cref{prop:upper_processes_def1,prop:upper_processes_def2}) related respectively to the partition $S^{+}$ and to the naive partition $\tilde{S}^+$.\\
For the numerical simulations, we used the set of parameters $\theta$ detailed in \Cref{tab:values_ex} for which the initial process $X$ and the upper-bounding process $Y$ have a stationary distribution (see \Cref{tab:asympto_behavior_parameters}) as opposed to the explosive process $\tilde{Y}$.

\begin{table}[h!]
\centering
\begin{tabular}[\linewidth]{ | c || c | c | c | c | c | c | c |}
\cline{2-8}
\multicolumn{1}{c|}{} & $b_1$ & $b_2$ & $\alpha$ & $\beta$ & $d_1$ & $d_2$ & $d_3$\\
\hline
$\theta$ & 1 & $\frac{5}{2}$ & $\frac{5}{2}$ & 2 & $\frac{5}{2}$ & $\frac{5}{2}$ & 3 \\
\hline
\end{tabular}
\caption{Parameter values used for the numerical illustrations.}\label{tab:values_ex}
\end{table}

The initial condition is fixed to $X_0 = (15, 5, 5)$ that corresponds to $Y_0 = 40$ and $\tilde{Y}_0 = 25$.\\

\begin{figure}[h!]
\centering
\includegraphics[scale= 0.55]{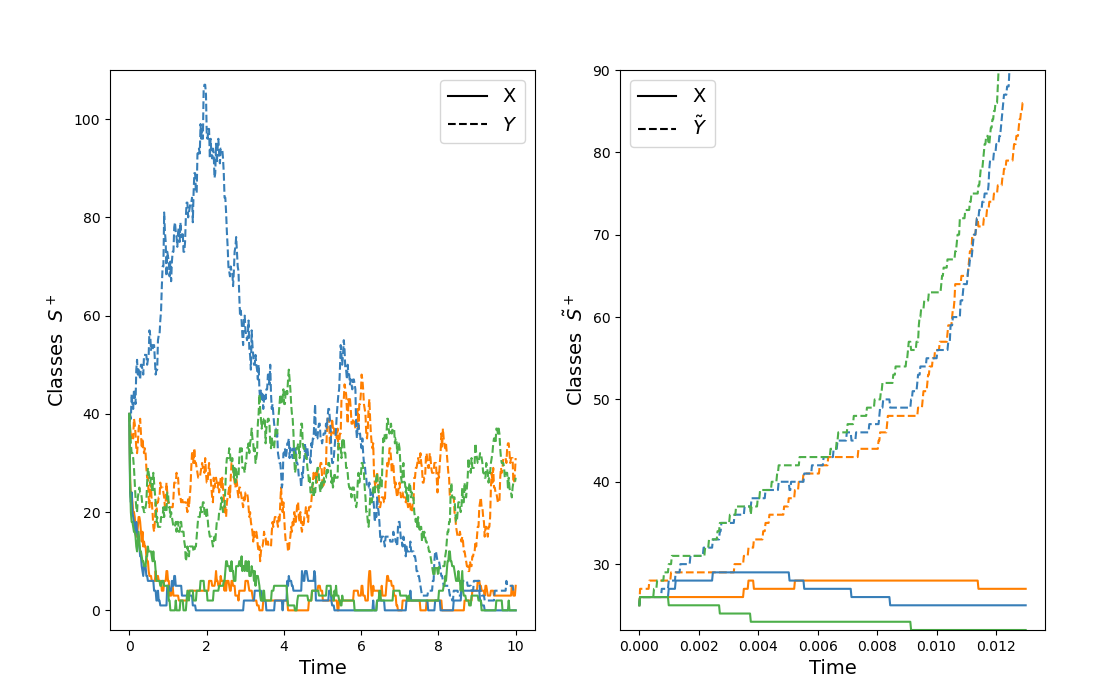}
\caption{Three trajectories of the coupling process for the two sets of classes $S^+$ (left panel) and $\tilde{S}^{+}$ (right panel). On each panel, we plot the trajectory of the upper-bounding process (second component of the coupling obtained with \Cref{thm:main_upper-bound}) in solid lines and the class of the original process in dashed lines with respect to time.}\label{fig:coupl_upper}
\end{figure}

\Cref{fig:coupl_upper} illustrates the difference in behavior at infinity between the two upper-bounding processes: the existence of a stationary distribution for $Y$ compared to the explosion of process $\tilde{Y}$. The three trajectories of $\tilde{Y}$ gradually move away from the initial process $X$ (see right-hand panel) and can only be stochastically simulated for very small time ($T_f = 0.013$ here) due to the explosion of the process. This is due to the quadratic birth rate of $\tilde{Y}$ (see \Cref{prop:upper_processes_def2}). On the contrary, the trajectories of $Y$ stabilise and remain close to $X$ (see left-hand panel).

\subsection{Truncation error}\label{ssec:truncation_error}

\subsubsection{Theoretical result}

In principle, the time evolution of the probability distribution of a CTMC can be computed by solving the corresponding master equation. However, for most reaction networks the state space is countably infinite and solving the master equation is impossible without first truncating the state space \cite{munsky_finite_2006}. Truncating the state space gives rise to a linear system of ordinary differential equations whose solution approximates the solution of the full master equation. In the case of a multi-dimensional CTMC, even after truncation the state space is often still large leading to a time-consuming numerical solving of this linear system. It is therefore impossible to test different truncations when we want to guarantee a given approximation error due to the truncation.\\

In this section, we show how the bounding processes can be used to guide the choice of truncation and to quantify the approximation error. Since the bounding processes are one-dimensional, they are easy to analyse and simulate so our approach allows us to deal with high-dimensional initial CTMC without any numerical problems.\\

Let $\Omega$ be the truncated state space and $T_f$ the final time. The truncation error is defined as the probability that the process of interest $X$ will exit $\Omega$ over the whole time interval $[0,T_f]$,
$$ e_{trunc}(\Omega) := \mathbb{P}\left(\exists ~t \leq T_f,~ X_t \notin \Omega \right).$$
The goal of this section is to use the upper-bounding process constructed in \Cref{sec:upper} to bound this truncation error. For the sake of simplicity, we will only consider truncated state spaces $\Omega$ defined through the classes, i.e for any $N \in \mathbb{N}$,
$$ \Omega_N = \bigcup_{\ell=0}^N S_{\ell}.$$
With \Cref{thm:main_upper-bound}, we can construct a one-dimensional continuous-time Markov chain $Y_t$ on $\mathbb{N}$ such that
$$  \forall t \geq 0,~ \cl{X_t} \leq Y_t.$$
We have a first obvious bound given by
\begin{equation*}
e_{trunc}(\Omega_N) \leq \mathbb{P}\left(\exists ~t \leq T_f,~ Y_t > N \right).
\end{equation*}
It is possible to compute an upper bound of $\mathbb{P}\left(\exists ~t \leq T_f,~ Y_t > N \right)$ in function of the solution to the CME associated with $Y$. In practice, only a truncated version $\tilde{p}^M$of the CME can be computed. For a given $M \in \mathbb{N}$ with $M \geq N$, we define it as follows:
\begin{equation*}
\forall t \geq 0,~ \forall m \in \llbracket 0, M \rrbracket,~ \frac{\d{\tilde{p}_{m}^M} }{\d t}(t) = \sum_{\ell =0}^M \tilde{p}_{\ell}^M(t) \tilde{Q}_{\ell,m},
\quad
\tilde{p}_m^M(0)= \mathbb{P}\big[ X_0 \in S_m \big].
\end{equation*}

Finally, we introduce the two following quantities,
\begin{itemize}
    \item the exit rate $F^{out}_{\Omega_N}(t)$ at time $t$ of the set $\Omega_N$,
    $$ F^{out}_{\Omega_N}(t) = \sum_{\ell =0}^N \sum_{m=N+1}^M \tilde{p}^M_{\ell}(t)\tilde{Q}_{\ell,m}, $$
    \item and the integral of $F^{out}_{\Omega_N}$ on $[0, T_f]$,
    $$ F_{\Omega_N} = \int_0^{T_f} F^{out}_{\Omega_N}(s) \d s. $$
\end{itemize}
Our next theorem gives an upper bound of $e_{trunc}(\Omega_N)$ in function of $\tilde{p}^M$.

\begin{thm}\label{thm:trunc_error}
Let $(M,N) \in \mathbb{N}^2$ such that $M \geq N$. Then the truncated error is upper bounded by
\begin{equation}
\label{eq:bound_error_truncation}
e_{trunc}(\Omega_N) \leq
\underbrace{\left( 1- \tilde{p}^{M}_{0:M}(T_f)\right) 
+
\tilde{p}^{M}_{N+1:M}(0)}_{=: C(N)}
+
 F_{\Omega_N}.
\end{equation}
So for a fixed error $\epsilon$, the minimal state space that guarantees a truncation error smaller than $\epsilon$ is $\Omega_{\hat{N}}$ with
\begin{equation}\label{eq:esti_truncation}
\hat{N} = \min E(\epsilon),
\end{equation}
where $E(\epsilon) = \{N,~ 0\leq N\leq M \text{ s.t }~ C(N) + F_{\Omega_N} < \epsilon \}$.
\end{thm}

\begin{proof}
Let us call $\tilde{p}^N$ the solution to the CME associated with $Y$ on the truncated state space $\llbracket 0, N \rrbracket$, i.e.,
\begin{equation*}
\forall t \geq 0,~ \forall m \in \llbracket 0, N \rrbracket,~ \frac{\d{\tilde{p}_{m}^N} }{\d t}(t) = \sum_{\ell =0}^N \tilde{p}_{\ell}^N(t) \tilde{Q}_{\ell,m},
\quad
\tilde{p}_m^N(0)= \mathbb{P}\big[ X_0 \in S_m \big].
\end{equation*}
We have a probabilistic interpretation of $\tilde{p}_m^N$, given by
\begin{equation} \label{eq:trunc0}
\tilde{p}_m^N(t)
= \mathbb{P} \big[ Y_t = m \ \text{and} \ Y_s \leq N, \forall s \in [0,t]\big].
\end{equation}
Therefore we have
\begin{equation*}
e_{trunc}(\Omega_N)
\leq 1- \tilde{p}_{0:N}^N(T_f).
\end{equation*}
We need now to find a lower bound of the quantity $\tilde{p}_{0:N}^N(T_f)$. We have
\begin{align}
\tilde{p}_{0:N}^N(T_f)
- \tilde{p}_{0:M}^M(T_f)
={} & \tilde{p}_{0:N}^N(0)-\tilde{p}_{0:M}^M(0)
+ \int_0^{T_f}
\frac{\text{d}}{\text{d}t}
\Big( \tilde{p}_{0:N}^N(t)-\tilde{p}_{0:M}^M(t) \Big)
dt \notag \\
= {} & -\tilde{p}_{N+1:M}^M(0)
+ \int_0^{T_f}
\Bigg[
\underbrace{
\sum_{\ell= 0}^N \Big( \tilde{p}_\ell^N(t) \sum_{m=0}^N \widetilde{Q}_{\ell,m} \Big)
-
\sum_{\ell= 0}^M \Big( \tilde{p}_\ell^M(t) \sum_{m=0}^M \widetilde{Q}_{\ell,m} \Big)
}_{=:\phi(t)}
\Bigg]
dt. \label{eq:trunc1}
\end{align}
We have the following inequality, for any $m \in \llbracket 0, N \rrbracket$:
\begin{equation*}
\tilde{p}_m^N(t)
= \mathbb{P} \big[ Y_t = m \ \text{and} \ Y_s \leq N, \forall s \in [0,t]\big]
\leq
\mathbb{P} \big[ Y_t = m \ \text{and} \ Y_s \leq M, \forall s \in [0,t]\big]
= \tilde{p}_m^M(t).
\end{equation*}
As $\sum_{m=0}^N \widetilde{Q}_{\ell,m} \leq 0$, we deduce that
\begin{equation*}
\phi(t)
\geq
\sum_{\ell= 0}^N \Big( \tilde{p}_\ell^M(t) \sum_{m=0}^N \widetilde{Q}_{\ell,m} \Big)
-
\sum_{\ell= 0}^M \Big( \tilde{p}_\ell^M(t) \sum_{m=0}^M \widetilde{Q}_{\ell,m} \Big).
\end{equation*}
Re-arranging the right-hand side, we obtain that $\phi(t) \geq \phi_1(t) + \phi_2(t)$, where
\begin{equation*}
\phi_1(t)= \sum_{\ell=0}^N
\Big[
\tilde{p}_\ell^M(t)
\Big(
\sum_{m=0}^N \widetilde{Q}_{\ell,m}- \sum_{m=0}^M \widetilde{Q}_{\ell,m}
\Big)
\Big]
\quad \text{and} \quad
\phi_2(t)= \sum_{\ell=N+1}^M \Big( \tilde{p}_\ell^M(t) \sum_{m=0}^M \widetilde{Q}_{\ell,m} \Big).
\end{equation*}
Obviously, we have
\begin{equation} \label{eq:trunc2}
\phi_1(t)
= -\sum_{\ell=0}^N \tilde{p}_{\ell}^M(t) \Big(
\sum_{m=N+1}^M \widetilde{Q}_{\ell,m}
\Big)
\quad \text{and} \quad
\phi_2(t) \leq 0.
\end{equation}
Combining \eqref{eq:trunc1} with \eqref{eq:trunc2}, we obtain that
\begin{equation*}
\tilde{p}_{0:N}^N(T_f)
 - \tilde{p}_{0:M}^M(T_f)
= -\tilde{p}_{N+1:M}^M(0)
- \sum_{\ell=0}^N \sum_{m=N+1}^M \int_0^{T_f} \tilde{p}_{\ell}^M(t) 
\widetilde{Q}_{\ell,m}
dt,
\end{equation*}
from which \eqref{eq:bound_error_truncation} follows, thanks to \eqref{eq:trunc0}.
\end{proof}

\Cref{thm:trunc_error} can be used to devise a method to determine for a fixed error $\epsilon$ the smallest truncated state space $\Omega_{\hat{N}}$ by only simulating the process $Y$ as follows:
\begin{itemize}
\item Calculate the transition-rate matrix $\tilde{Q}$ of the upper-bounding process $Y$ using \Cref{prop:optimal_upper},
\item Solve the master equation for the process $Y$ on $[0,T_f]$ on the finite space $\llbracket 0,M \rrbracket$ for $M$ such that
$$ 1- \tilde{p}^{M}_{0:M}(T_f) \leq \epsilon. $$
\item Compute the size $\hat{N}$ of the truncated state space using \Cref{eq:esti_truncation}.
\end{itemize}

Finding a value for $M$ such that the quantity $1- \tilde{p}^{M}_{0:M}(T_f)$ is smaller than $\epsilon$ is easy because $Y$ is a one-dimensional CTMC so a large value for $M$ can be picked directly. Just directly choosing a very large truncation for $X$ would not be feasible since the high-dimensional state space would make solving the master equation too costly in terms of calculation time.
Moreover, the same value of $M$ can be used to compute the minimal truncated state space for different values of $\epsilon$ which avoids having to solve a new master equation.

\subsubsection{Application of the method to the toy model}

In this subsection, we illustrate the method derived from \Cref{thm:trunc_error} on the estimation of the error coming from the truncation of the state space with the upper-bounding process used in the previous section.\\

Let us recall the following set of classes defined by \Cref{eq:def_upper_classes},
$$ \forall \ell \geq 0,~ S^{+}_{\ell} = \{ 2x_1 + x_2+x_3= \ell \}, $$
for which the optimal upper-bounding process $Y$ is defined through its transition-rate matrix $Q^+$ by \Cref{prop:upper_processes_def1}.

For the numerical simulations, we used the set of parameters $\theta$ detailed in \Cref{tab:values_ex}, $M = 2000$ and $T_f=4$. With the same notation as in \Cref{thm:trunc_error}, we denote by $\tilde{p}^{M}$ the solution of the master equation of $Y$ truncated on $[0,M]$ until time $T_f$. The initial condition $X_0$ of the original process is set to $(30,10,10)$ which corresponds to $\tilde{p}^{M}(0)=\delta_{80}$.\\

The upper bound $E^T$ of the truncation error, defined by \Cref{eq:bound_error_truncation}, is expressed for any $0\leq N\leq M$ as
$$ E^T(N) = \left( 1- \tilde{p}^{M}_{0:M}(T_f)\right) + \delta_{N \leq 80} + \int_0^{T_f} \sum_{x \in \Omega_N} \sum_{y \notin \Omega_N} \tilde{p}^{M}_x(t) Q^+(x,y), $$
with $\Omega_N = \bigcup\limits_{\ell=0}^N S^+_{\ell}$.\\

\begin{figure}[h!]
\centering
\includegraphics[scale= 0.7]{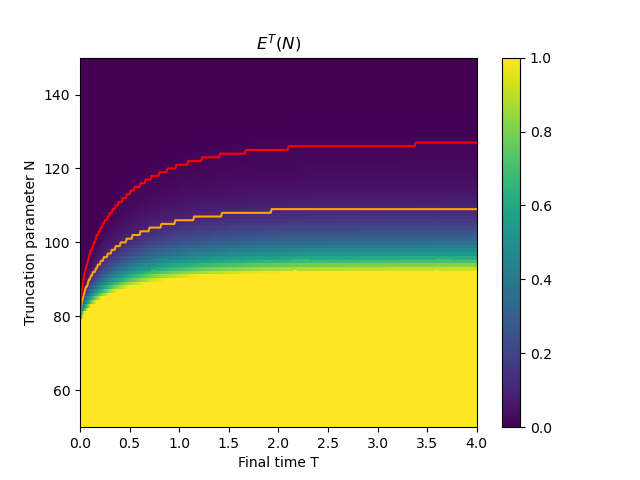}
\caption{Upper-bound $E^T(N)$ of the truncation error as a function of the final time $T$ and truncation parameter $N$ for the well-selected classes $S^+$. The $y$-axis represents the parameter $N$ chosen for truncation. The orange line (respectively the red line) represents the smallest $N$ obtained with \Cref{eq:esti_truncation} that leads to an error smaller than $0.1$ (resp. $0.01$).\\
Note that since $E^T(N)$ is an upper-bound of a probability, it could take values greater than $1$ even if we know that the estimated probability is lower than $1$. Knowing that values of $E^T(N)$ greater than 1 are set to 1 (yellow part of the graph).}\label{fig:error_trunc}
\end{figure}

The quantity $E^T(N)$ is plotted in \Cref{fig:error_trunc} for different final times $T$ (x-axis) and for different values of $N$ (y-axis). Note that the $E^T(N)$ is equal to $1$ for $N$ smaller than $80$ which is consistent with the initial condition being equal to $80$. The orange line (respectively the red line) represents the smallest $N$ for which the upper bound of the error is smaller than a fixed error $\epsilon$ equal to $0.1$ (resp. 0.01). Note that the same value of $M$ was used to compute the orange and red lines meaning that we solve the master equation links to Y once. The quantity $E^T(N)$ preserves the stationarity of the process since the orange and red lines does not go to infinity with time.\\
If we had chosen the classes $\tilde{S}^+$ (see \Cref{eq:def_upper_classes}) and the corresponding optimal upper-bounding process $\tilde{Y}$, we would not have obtained a converging truncation bound since $\tilde{Y}$ is explosive.

\section{Discussion}

Stochastic reaction networks are tremendously flexible and can serve as models to represent stochastic processes in various different scientific areas. However, in most applications, reaction networks comprise a multitude of interacting species, which implies that their dynamics follow a multi-dimensional continuous-time Markov chain that is notoriously difficult to analyze. In most cases, analytical solutions to the governing master equation are not available and even just stochastic simulation may be computationally too expensive to be practical.\\

In this paper, we have shown that despite these difficulties it may still be possible to establish properties of the CTMC such as its long-term behavior and the existence of a stationary distribution. To achieve this, we defined a partition of the state space, that we referred to as classes throughout the manuscript, and showed that the original CMTC transferred to the classes can be upper and lower bounded by one-dimensional birth death processes that are easy to analyze analytically and numerically. We proved that properties of the bounding processes can be used to gain insights on the original CTMC. For instance, for irreducible CTMCs, the existence of a stationary distribution for the upper-bounding process guarantees that the original CTMC also admits a stationary distribution. These results allow us to gain some level of understanding about complex intractable CTMCs from the study of easily tractable one-dimensional bounding processes.\\

The quality and accuracy of so-obtained results will generally depend on how tightly the bounding processes bound the original CTMC. We have shown how, for a given partition of the state space of the original CTMC, one can define the rates of the bounding processes that are guaranteed to lead to optimal bounds (\Cref{subsec:construction_Qtilde,subsec:construction_Qtilde_lower}). However, whether or not these optimal bound are practically useful still depends on how exactly the classes that partition the state space are chosen. Maybe the most natural choice is to define classes by the total number of individuals/molecules of all interacting species that are present in the system. We have shown, however, that this is a choice that is agnostic to the structure of the reaction network and may lead to bad result. For instance, for the case study in \Cref{sec:appli}, this naive choice of classes leads to an (optimal) upper-bounding process that is generally explosive despite the fact that the original CTMC, depending on the parameter values that determine the reaction rates, may admit a stationary distribution. On the other hand, as we have shown in \Cref{subsec:asymptotic_behavior}, a better choice of the classes that respects the structure of the reaction network allows one to use the lower and upper-bounding processes to determine the asymptotic behavior of the original process in different regions of parameter space.\\

In addition to using bounding processes to establish properties such as recurrence or existence of a stationary distribution, we have shown how the upper-bounding process can be used to guide the choice of finite state projections for the numerical analysis of CTMCs with countably infinite state spaces. Truncation of the state space is often the only possible approach for numerically approximating the solution of the master equation of the process. Thereby, approximation errors depend on how the truncation is chosen and increase with time such that it is often difficult to say \textit{a priori} how large the truncation needs to be to achieve a given error. We have shown in \Cref{ssec:truncation_error} that the error that a given truncation will create in the approximation of the solution of the master equation of the original CTMC can be bounded based on the solution of the master equation of the upper-bounding process, which is easy to approximate numerically since the upper-bounding process is one-dimensional. As we have shown in our case study, this allows us to straightforwardly calculate necessary truncation sizes to achieve given error guarantees for any desired time horizon (\Cref{fig:error_trunc}).\\

In sum, we believe that our scheme for constructing one-dimensional bounding processes for multi-dimensional CTMCs, together with our proofs on how properties of these bounding processes carry over to the original CTMC, will help researchers to gain a better understanding of the properties of the complex stochastic reaction networks that emerge from applications in biology and other fields.\\

\textbf{Acknowledgments:} This work was supported by the ANR (Opt-MC project, ANR-22-CE45-0019) and the European Research Council (ERC-2022-STG, BridgingScales, grant agreement 101075989).

\bibliographystyle{abbrv}
\bibliography{Manuscript_Ballif_Pfeiffer_Ruess}

\newpage
\begin{appendix}

\section{Technical calculations}

\begin{lem}\label{lem:technical_calculs}
Let $(a,b, c_1, c_2) \in \left(\mathbb{R}^+\right)^2 \times \mathbb{R}^2 $ and for $x \in \mathbb{R}^2$,
$$F(x) = ax_{1}^2 + bx_1 x_2 + c_1 x_1 +c_2x_2. $$
The goal of this lemma is to compute the minimum of the function $F$ on a subset of the upper right quadrant of the plane. To this purpose, we define for any $(\ell, \alpha_1, \alpha_2) \in \mathbb{N}_{*}^3$,
$$ M_{\ell} = \min_{x \in D_{\ell}} F(x), $$
with $D_{\ell} = \{ x \in \mathbb{R}_{+}^2,~ \alpha_1 x_1 + \alpha_2 x_2 \leq \ell \}$.\\

Then the elements of $D_{\ell}$ where the minimum of $F$ can be reached (infinite values are allowed if the denominator is null) are
\begin{equation*}
\begin{array}{llll}
p_1 = \left(-\frac{c_2}{b},~ \frac{2ac_2-bc_1}{b^2} \right), & p_2 = \left(0,0 \right), & p_3 = \left(0, \frac{\ell}{\alpha_2} \right),\\
p_4 = \left(-\frac{c_1}{2a},0 \right), & p_5 = \left(\frac{\ell}{\alpha_1}, 0 \right), & p_6 = \left( -d_1 \ell + d_2, \frac{1+d_1 \alpha_1}{\alpha_2}\ell -\frac{\alpha_1}{\alpha_2}d_2 \right),
\end{array}
\end{equation*}
with $d_1 = \frac{b}{2(a\alpha_2-b\alpha_1)}$ and $d_2 = \frac{\alpha_1 c_2-\alpha_2 c_1}{2(a\alpha_2-b\alpha_1)}$.\\

So we have, with $\mathcal{P} = \{p_i,~ 1 \leq i \leq 6\}$,
$$ M_{\ell} = \min_{p \in \mathcal{P} \cap D_{\ell}} F(p). $$
\end{lem}

\begin{proof}
Let $\bar{D_{\ell}}$ be the interior of $D_{\ell}$ and $\overset{\circ}{D_{\ell}}$ the boundary of $D_{\ell}$, that is
$$ \bar{D_{\ell}} = \{ x \in \mathbb{R}_{+}^2,~ \alpha_1 x_1 + \alpha_2 x_2 < \ell \}, $$
$$ \overset{\circ}{D_{\ell}} = \left\{ (0,x_2),~ 0 \leq x_2 \leq \frac{\ell}{\alpha_2} \right\} \cup \left\{ (x_1, 0),~ 0 \leq x_1 \leq \frac{\ell}{\alpha_1} \right\} \cup \{ x \in \mathbb{R}_{+}^2,~ \alpha_1 x_1 + \alpha_2 x_2 = \ell \}. $$
Two cases must be considered.
\begin{enumerate}
\item For $x \in \bar{D_{\ell}}$, the minimum $M_{\ell}$ is reached on $x$ if $x$ is a critical point that is
$$ \frac{\partial F}{\partial x_1}(x) = 2a x_1 + bx_2 +c_1 = 0, $$
$$ \frac{\partial F}{\partial x_2}(x) = bx_1 +c_2 = 0. $$
The only solution of this system is $p_1 = \left(-\frac{c_2}{b},~ \frac{2ac_2-bc_1}{b^2} \right)$.
\item For $\overset{\circ}{D_{\ell}}$, we will consider separately each set in the definition of $\overset{\circ}{D_{\ell}}$.
\begin{itemize}
\item For $x \in \left\{ (0,x_2),~ 0 \leq x_2 \leq \frac{\ell}{\alpha_2} \right\}$, $F(x) = c_2x_2$.\\
The minimum is reached for $x_2 = 0$ or $x_2 =\frac{\ell}{\alpha_2}$ according to the sign of $c_2$. Thus, the corresponding points in the plane are $p_2 = \left(0,0 \right)$ and $p_3 = \left(0, \frac{\ell}{\alpha_2} \right)$.
\item For $x \in \left\{ (x_1,0),~ 0 \leq x_1 \leq \frac{\ell}{\alpha_1} \right\}$, $F(x) = a x_1^2 + c_1x_1 = x_1 (ax_1 + c_1)$.\\
If $a > 0$, the minimum is reached for $x_1 = -\frac{c_1}{2a}$, which corresponds to $p_4 = \left(-\frac{c_1}{2a},0 \right)$.\\
If $a = 0$, the minimum is reached for $x_1 = 0$ or $x_1 = \frac{\ell}{\alpha_1}$, which corresponds to $p_2$ and $p_5$.
\item For $x \in \{ x \in \mathbb{R}_{+}^2,~ \alpha_1 x_1 + \alpha_2 x_2 = \ell \}$,
\begin{align*}
F(x) &= ax_1^2 + b\frac{x_1}{\alpha_2}(\ell - \alpha_1 x_1) + c_1 x_1 + \frac{c_2}{\alpha_2}(\ell - \alpha_1 x_1)\\
&= \left(a-b\frac{\alpha_1}{\alpha_2} \right)x_1^2 + \left(\frac{b}{\alpha_2}\ell + c_1 -c_2\frac{\alpha_1}{\alpha_2} \right)x_1 + \frac{c_2}{\alpha_2}\ell .
\end{align*}
If $\left(a-b\frac{\alpha_1}{\alpha_2} \right) \leq 0$, the minimum is reached for $x_1 = 0$ or $x_1 = \frac{\ell}{\alpha_1}$, which corresponds to $p_3$ and $p_5$.\\
If $\left(a-b\frac{\alpha_1}{\alpha_2} \right) > 0$, an easy computation shows that the minimum is reached for $x_1 = -d_1 \ell +d_2$ which corresponds to $p_6$.
\end{itemize}
\end{enumerate}
\end{proof}

\subsection{Proof of \Cref{prop:matrix_lower}}\label{ssec:app_1}

The proof is composed of two steps: first the calculation of the quantities $\hat{f}^-$ and $\hat{f}^+$ defined by \Cref{eq:def_f_lower}, then with \Cref{eq:optimal_lower}, we obtain the expressions of $\hat{U}^-$ and $\hat{U}^+$.\\

In this proposition, the set of classes used is
$$ \forall \ell \in \mathbb{N}, ~~ S^{-}_{\ell} = \{x \in \mathbb{N}^3,~ 2x_1 + 2x_2 + 5x_3 = \ell \}. $$

Starting from a state $x$ in class $\ell \geq 5$,
\begin{itemize}
\item after the reaction $X_3 \longrightarrow \emptyset$, the process ends in the class $\ell -5$,
\item after the reaction $X_1 \longrightarrow \emptyset$ or $X_2 \longrightarrow \emptyset$, the process ends in the class $\ell -2$,
\item after the reaction $X_1+X_2 \longrightarrow X_3$, the process ends in the class $\ell +1$,
\item after the reactions $\emptyset \longrightarrow X_1$ and $2X_1 \longrightarrow 3X_2$, the process ends in the class $\ell +2$.
\end{itemize}

Thus we have, with \Cref{eq:def_f_lower},
$$ \hat{f}_{\ell,m}^- = \begin{cases}
\begin{array}{ll}
0 & \text{ if $m < \ell-5$},\\
\max\limits_{x \in S^{-}_{\ell}} d_3 x_3 = \frac{d_3}{5}\ell  & \text{ if $\ell-5 \leq m \leq \ell-3$},\\
\max\limits_{x \in S^{-}_{\ell}} (d_1x_1 + d_2x_2 + d_3 x_3) = \max\Big(\frac{d_1}{2}, \frac{d_2}{2}, \frac{d_3}{5} \Big)\ell  & \text{ if $\ell-2 \leq m < \ell$},\\
\end{array}
\end{cases}
$$
$$ \hat{f}_{\ell,m}^+ = \begin{cases}
\begin{array}{ll}
\min\limits_{x \in S^{-}_{\ell}} \big(b_1(x_2+x_3)+b_2+\alpha x_1(x_1-1) + \beta x_1x_2\big) & \text{ if $m = \ell+1$},\\
\min\limits_{x \in S^{-}_{\ell}} \big(b_1(x_2+x_3)+b_2+\alpha x_1(x_1-1) \big)  & \text{ if $m = \ell+2$},\\
0 & \text{ if $\ell +3 \leq m$}.\\
\end{array}
\end{cases}
$$
Moreover,
\begin{align*}
\hat{f}_{\ell,\ell+2}^+ &= \min\limits_{x \in S^{-}_{\ell}} \big(b_1(x_2+x_3)+b_2+\alpha x_1(x_1-1) \big)\\
&= \min\limits_{x \in D_{\ell}} \left(\alpha x_{1}^2 - \Big(\alpha + \frac{2}{5}b_1 \Big)x_1 + \frac{3}{5}b_1 x_2\right) + b_2 + \frac{b_1}{5} \ell ~~\text{ where $D_{\ell} = \{x \in \mathbb{R}_{+}^2,~ 2x_1+2x_2 \leq \ell \}$}
\end{align*}
Using \Cref{lem:technical_calculs}, for $\ell$ large enough the minimum is reached in $p_4 = \left(\frac{\alpha+\frac{2}{5}b_1}{2\alpha}, 0 \right)$ and
$$ \hat{f}_{\ell,\ell+2}^+ = \frac{b_1}{5}\ell +  A(\theta),$$
with $A(\theta) = b_2 -\frac{(\alpha+ \frac{2}{5}b_1)^2}{4\alpha}$.\\

With the same computation, $ \hat{f}_{\ell,\ell+1}^+ = \frac{b_1}{5}\ell +  A(\theta)$.\\

The result follows from \Cref{eq:optimal_lower}.

\subsection{Proof of \Cref{prop:upper_processes_def1}}\label{ssec:app_2}

The proof is composed of two steps: first the calculation of the quantities $f^-$ and $f^+$ defined by \Cref{eq:def_f_upper}, then with \Cref{eq:optimal_upper}, we obtain the expressions of $U^-$ and $U^+$ from which we deduce the expression of the matrix $Q^+$.\\

In this proposition, the set of classes used is
$$ \forall \ell \in \mathbb{N}, ~~ S^{+}_{\ell} = \{x \in \mathbb{N}^3,~ 2x_1 + x_2 + x_3 = \ell \}. $$

Starting from a state $x$ in class $\ell \geq 2$,
\begin{itemize}
\item after the reactions $X_1 \longrightarrow \emptyset$ or $X_1+X_2 \longrightarrow X_3$, the process ends in the class $\ell -2$,
\item after the reactions $2X_1 \longrightarrow 3X_2$, $X_2 \longrightarrow \emptyset$ or $X_3 \longrightarrow \emptyset$, the process ends in the class $\ell -1$,
\item after the reaction $\emptyset \longrightarrow X_1$, the process ends in the class $\ell +2$.
\end{itemize}

Thus we have, with \Cref{eq:def_f_upper},
$$ f_{\ell,m}^- = \begin{cases}
\begin{array}{ll}
0 & \text{ if $m < \ell-2$},\\
\min\limits_{x \in S^{+}_{\ell}} (d_1 x_1 + \beta x_1 x_2) = 0  & \text{ if $m = \ell-2$},\\
\min\limits_{x \in S^{+}_{\ell}} \left(\alpha x_1(x_1 - 1) + \sum\limits_{i=1}^3 d_ix_i \right)  & \text{ if $m = \ell-1$},\\
\end{array}
\end{cases}
$$
$$ f_{\ell,m}^+ = \begin{cases}
\begin{array}{ll}
\max\limits_{x \in S^{+}_{\ell}} \big(b_1 (x_2+x_3)+b_2\big) = b_1\ell+b_2 & \text{ if $\ell+1 \leq m \leq \ell+2$},\\
0 & \text{ if $\ell +3 \leq m$}.\\
\end{array}
\end{cases}
$$
Moreover,
\begin{align*}
f_{\ell,\ell-1}^- &= \min\limits_{x \in S^{+}_{\ell}} \left(\alpha x_1(x_1 - 1) + \sum\limits_{i=1}^3 d_ix_i \right)\\
&= \min\limits_{x \in D_{\ell}} \left(\alpha x_{1}^2 + \Big(d_1-\alpha -2 d_3 \Big)x_1 + (d_2-d_3)x_2\right) + d_3 \ell, ~~\text{ where $D_{\ell} = \{x \in \mathbb{R}_{+}^2,~ 2x_1+x_2 \leq \ell \}$.}
\end{align*}
Using \Cref{lem:technical_calculs}, for $\ell$ large enough the minimum is reached in
$$p_4 = \left(\frac{\alpha+2d_3-d_1}{2\alpha}, 0 \right) \text{~ or ~} p_6 = \left(\ell - \frac{\alpha+2d_2-d_1}{\alpha}, 0 \right),$$ and
$$ f_{\ell,\ell-1}^- = \min(d_2,d_3)\ell -  C(\theta), $$
with $C(\theta) = \frac{(\alpha+2\min(d_2,d_3)-d_1)^2}{4\alpha}\mathds{1}_{\alpha+2\min(d_2,d_3)-d_1 \geq 0}(\theta)$.\\

The result follows directly from \Cref{eq:optimal_upper}.

\subsection{Proof of \Cref{prop:upper_processes_def2}}\label{ssec:app_3}

The proof is composed of two steps: first the calculation of the quantities $f^-$ and $f^+$ defined by \Cref{eq:def_f_upper}, then with \Cref{eq:optimal_upper}, we obtain the expressions of $U^-$ and $U^+$ from which we deduce the expression of the matrix $\tilde{Q}^+$.\\

In this proposition, the set of classes used is
$$ \forall \ell \in \mathbb{N}, ~~ \tilde{S}^{+}_{\ell} = \{x \in \mathbb{N}^3,~ x_1 + x_2 + x_3 = \ell \}. $$

Starting from a state $x$ in class $\ell \geq 1$,
\begin{itemize}
\item after the reactions $X_1+X_2 \longrightarrow X_3$, $X_1 \longrightarrow \emptyset$, $X_2 \longrightarrow \emptyset$ or $X_3 \longrightarrow \emptyset$, the process ends in the class $\ell -1$,
\item after the reactions $\emptyset \longrightarrow X_1$ or $2X_1 \longrightarrow 3X_2$, the process ends in the class $\ell +1$.
\end{itemize}

Thus we have, with \Cref{eq:def_f_upper},
$$ f_{\ell,m}^- = \begin{cases}
\begin{array}{ll}
0 & \text{ if $m < \ell-1$},\\
\min\limits_{x \in \tilde{S}^{+}_{\ell}} \left(\beta x_1 x_2 + \sum\limits_{i=1}^3 d_ix_i \right) = \min(d_1,d_2,d_3)\ell  & \text{ if $m = \ell-1$},\\
\end{array}
\end{cases}
$$
$$ f_{\ell,m}^+ = \begin{cases}
\begin{array}{ll}
\max\limits_{x \in \tilde{S}^{+}_{\ell}} \Big(\alpha x_1 (x_1-1) + b_1 (x_2+x_3)+b_2\Big) = \alpha \ell^2 -\alpha \ell+b_2 & \text{ if $m = \ell+1$},\\
0 & \text{ if $\ell +2 \leq m$}.\\
\end{array}
\end{cases}
$$
The result follows directly from \Cref{eq:optimal_upper}.

\end{appendix}

\end{document}